\documentclass[11pt]{article}\textwidth 160mm\textheight 235mm
\usepackage{amsfonts}
\usepackage{amssymb}
\usepackage{graphicx}
\usepackage{amsmath}
\usepackage{amsthm}
\usepackage{enumitem}
\usepackage{tikz}
\usepackage{mathrsfs}
\usepackage{cancel}
\usepackage{todonotes}
\usepackage{cases}
\usetikzlibrary{matrix}
\usetikzlibrary{arrows,shapes}
\usepackage{cite}
\usepackage{bm}
\usepackage{hyperref}
\hypersetup{colorlinks=true, urlcolor=blue}
\expandafter\let\expandafter\oldproof\csname\string\proof\endcsname
\let\oldendproof\endproof
\renewenvironment{proof}[1][\proofname]{%
  \oldproof[\ttfamily \scshape \bf #1. ]%
}{\oldendproof}
\def\A{{\mathscr{A}}}
\def\U{{\cal U}}

\def\O{{\cal O}}
\def\C{{\cal C}}
\def\M{{\cal M}}
\def\P{{\bf P}}
\def\cP{{\cal P}}
\def\Z{{\cal Z}}
\def\B{\mathbb{B}}
\def\R{{\mathbb R}}
\def\oR{\overline{\R}}
\def\N{{\rm I\!N}}
\def\oi{{i_0}}

\def\ox{{\bar x}}

\def\oy{{\bar y}}

\def\ov{{\bar v}}

\def\op{{\bar p}}

\def\X{{\bf X}}
\def\Y{{\bf Y}}
\def\Prob{{\mathbb P}}

\def \b{{\}_{k\in\N}}}

\def\L{{\mathcal{L}}}
\def\D{{\mathscr{D}}}
\def\I{{\cal I}}

\def\xk{x^k}

\def\uk{u^k}
\def\vk{v^k}

\def\what{\widehat}

\def\tto{\rightrightarrows}

\def\prox{{\rm prox}\,}

\def\tto{\rightrightarrows}

\def\sub{\partial}

\def\ra{\rangle}
\def\la{\langle}
\def\ve{\varepsilon}
\def\omu{\bar{\mu}}

\def\olm{{\bar\lambda}}

 \def\para{{\rm par}\,}
\def\dd{\delta}

\def\th{\theta}
\def\vt{\vartheta}

\def\ph{\varphi}
\def\opsi{{\overline{\psi}}}
\def\hph{\widehat\varphi}

\def\toset_#1{\xrightarrow{#1}}
\DeclareMathOperator*{\mini}{minimize\;}
\DeclareMathOperator*{\argmin}{argmin\;}
\DeclareMathOperator*{\argmax}{argmax\;}
\DeclareMathOperator*{\E}{\mathbb{E}}
\def\Z{{\sf Z}}
\def\d{{\rm d}}
\def\dist{{\rm dist}}
\def\spann{{\rm span}\,}
\def\rge{{\rm rge\,}}

\def\Var{\mathbb{V}{\rm ar}\,}

\def\ri{{\rm ri}\,}

\def\tr{{\rm tr}\,}
\def\inte{{\rm int}\,}
\def\gph{{\rm gph}\,}
\def\epi{{\rm epi}\,}
\def\dim{{\rm dim}\,}
\def\dom{{\rm dom}\,}

\def\ker{{\rm ker}\,}
\def\sign{{\rm sgn}\,}

\def\aff{{\rm aff}\,}

\def\sm{\hbox{${1\over 2}$}}

\def\rN{{\what{N}}}

\begin{document}
\vspace*{0.5in}
\begin{center}
{\bf A FRESH LOOK INTO VARIATIONAL ANALYSIS OF\\ $\C^2$-PARTLY SMOOTH FUNCTIONS}\\[2ex]
NGUYEN T. V. HANG\footnote{School of Physical and Mathematical Sciences, Nanyang Technological University, Singapore 639798. Current address: University of Engineering and Technology, Vietnam National University, Hanoi, Vietnam (ntvhang@vnu.edu.vn). Research of this author is partially supported by  Singapore National Academy of Science via SASEAF Programme under the grant RIE2025 NRF International Partnership Funding Initiative.} 
and   EBRAHIM SARABI\footnote{Department of Mathematics, Miami University, Oxford, OH 45065, USA (sarabim@miamioh.edu). Research of the second author is partially supported by the U.S. National Science Foundation  under the grant DMS 2108546.
}
\end{center}
\vspace*{0.05in}

\small{\bf Abstract.} 
$\C^2$-partial smoothness of functions has been  an important subject of   research in optimization, on both theoretical and algorithmic aspects, since it was first introduced by Lewis in \cite{Lew02}.
Our work aims at providing a fresh variational analysis viewpoint on the class of $\C^2$-partly smooth functions.
Namely, we explore the relationship between  $\C^2$-partial smoothness and  strict twice epi-differentiability and demonstrate that functions from the latter class are always strictly twice epi-differentiable. On the other hand, we provide two examples to show that the opposite conclusion does not hold in general.  
As a consequence of our analysis, we calculate the second subderivative of
$\C^2$-partly smooth functions. 
Applications to stability analysis of related generalized equations involving a general perturbation and to asymptotic analysis of the well-known sample average approximation method for stochastic programs with $\C^2$-partly smooth regularizers are also given.
 \\[1ex]
{\bf Keywords.} $\C^2$-partial smoothness, prox-regularity, strict proto-differentiability, differential stability, (strong) metric regularity, sample average approximation method. \\[1ex]
{\bf Mathematics Subject Classification (2000)} 49J52, 65K10, 90C15, 90C31

\newtheorem{Theorem}{Theorem}[section]
\newtheorem{Proposition}[Theorem]{Proposition}
\newtheorem{Lemma}[Theorem]{Lemma}
\newtheorem{Corollary}[Theorem]{Corollary}

\numberwithin{equation}{section}

\theoremstyle{definition}
\newtheorem{Definition}[Theorem]{Definition}
\newtheorem{Example}[Theorem]{Example}
\newtheorem{Remark}[Theorem]{Remark}
\newtheorem{Assumption}[Theorem]{Assumption}

\renewcommand{\thefootnote}{\fnsymbol{footnote}}

\normalsize

\section{Introduction}\label{sect:intro}

This paper is a continuation of our recent work \cite{HaS23} on strict proto-differentiability of subgradient mappings and its applications in stability analysis of generalized equations, with an emphasis on {\em $\C^2$-partly smooth} functions.
The notion of partial smoothness, first formalized in \cite{Lew02} by Lewis, expresses an underlying smooth structure of a nonsmooth function on a smooth manifold and  offers a powerful framework for both stability and algorithmic analysis of optimization problems.  Therefore, there  has been 
considerable efforts to study   various theoretical aspects (cf.  \cite{BuE20,DDL14,DIL16,Har04, HaL04, LeZ13, MiS03, Wyl19}) as well as   numerical advantages (cf. \cite{BIM23, BuE20,DHM06,Har04, LeW21,LFP14, LFP17, Wyl19})
of $\C^2$-partly smooth functions.

Partial smoothness is widely satisfied by important classes of functions in optimization and variational analysis, including the polyhedral and pointwise maximum functions \cite{Lew02}, $\C^2$-decomposable functions \cite{Sha03}, and spectral  functions \cite{Lew02, DDL14}.
Calculus rules for this property,  established in \cite[Section~4]{Lew02} and \cite[Chapter~3]{Har04}, furthermore prove the abundance of partly smooth functions in these areas.

Motivated by the aforementioned   applications, we plan to study new second-order variational properties of  $\C^2$-partly smooth functions. More precisely, we follow our recent developments in  \cite{HJS22, HaS22, HaS23}, where novel second-order variational properties, namely strict twice epi-differntiability and    strict proto-differentiability, of important classes of composite functions were brought to the surface.
We also paid special attention to  continuous differentiability  of the proximal mappings and characterizations of regularity properties of certain generalized equations.
Similar results about the proximal mappings were established for $\C^2$-partly smooth functions in \cite[Theorem~28]{DHM06} and \cite[Corollaries~4.4.2 and~4.4.3]{Wyl19}.
In addition, in \cite[Section~5]{HaS23}, strict twice epi-differentiability was verified for a subclass of $\C^2$-partly smooth functions, called therein  reliably $\C^2$-decomposable functions; see Example~\ref{ex:decomp} for the definition of this class of functions.
These inspire us to explore further the relationships between  $\C^2$-partial smoothness and  strict twice epi-differentiability of functions.
We should add here that strict twice epi-differentiability was defined   by Rockafellar and Poliquin in \cite{pr} and was further studied for prox-regular functions in \cite{pr2}. In particular, a characterization of this concept in term of epigraphical convergence of the second subderivative of prox-regular functions was achieved in \cite[Theorem~4.1]{pr2}, which will be leveraged in this paper to show  under a relative interior condition
for the chosen subgradient that $\C^2$-partly smooth functions are always strictly twice epi-differentiable in a neighborhood of a given point relative to the graph of subgradient mappings. Moreover, our proof allows us to calculate the second subderivative of $\C^2$-partly smooth functions. We also provide two examples to show that there are strictly twice epi-differentiable functions that are not $\C^2$-partly smooth. 
Not only does this enable us to extend the scope of our developments in \cite{HaS23} to a larger class of functions but also provides   a fresh viewpoint to investigate such important functions in optimization and variational analysis.
We then concentrate on stability analysis of a generalized equation related to subgradient mapping of a $\C^2$-partly smooth function under a general perturbation and provide sufficient conditions for its solution mapping to have a single-valued Lipschitz continuous graphical localization.  Using this, we  analyze asymptotic behavior of the sequence, generated by a sample average approximation (SAA) method, for certain regularized stochastic programs.

The rest of the paper is organized as follows. 
Section~\ref{sect:Pre} recaps definitions of important concepts used in this paper.
Section~\ref{sect:sTED} begins with reviewing the notion   of partial smoothness. 
We then present our main findings about strict twice epi-differentiability of $\C^2$-partly smooth functions and regularity properties of the corresponding subgradient mappings.
As important consequences, we establish next in this section $\C^1$-smoothness and active manifold identifiability of proximal mappings as well as differential stability for generalized equations involving a general perturbation.
The last section, Section~\ref{sect:asym}, concerns asymptotic analysis of an SAA method for stochastic programs with $\C^2$-partly smooth regularizers.

\section{Tools of Variational Analysis}\label{sect:Pre}
Throughout, suppose that $\X$ and $\Y$ are finite dimentional Hilbert spaces.
We denote by $\B$ the closed unit ball in the space in question and by $\B_r(x):= x+r\B$ the closed ball centered at $x$ with radius $r>0$.
Given a nonempty set $C\subset \X$,
the symbols $\ri C$, $\aff C$, $C^*$, and $\para C$  signify its relative interior, affine span, polar cone, and the linear subspace parallel to $\aff C$, respectively.
The indicator function $\delta_C$ of the set $C$ is defined by $\delta_C(x)= 0$ for $x\in C$ and $\delta_C(x)=\infty$ for $x\in \X\setminus C$.
We denote by $P_C$ the projection mapping onto $C$ and by $\dist(x, C)$ the distance between $x\in \X$ and the set $C$.
The domain, range, and graph of a set-valued mapping $F:\X\tto\Y$ are defined, respectively, by $\dom F=\{x\in \X\, \big|\, F(x)\neq\emptyset\}$, $\rge F = \{y\in \Y\, \big|\, \exists\, x\in \X \,\textrm {with }\, y\in F(x)\}$, and $\gph F = \{(x, y)\in \X\times \Y\, \big|\, y\in F(x)\}$.
Consider an extended-real-valued function $\varphi:\X \to \oR:=\R\cup \{\pm\infty\}$.
We denote by $\epi \varphi$ its epigraph set given by $\{(x, \alpha)\in \X\times \R\, \big|\, \alpha \geq \varphi(x)\}$. 
Assuming that $\varphi$ is twice differentiable at $\ox$ with Hessian $\nabla^2\varphi(\ox)$, we also use the same notation $\nabla^2\varphi(\ox)$ to denote the self-adjoint linear operator of $\X$ associated with the symmetric bilinear form $\nabla^2\varphi(\ox):\X\times\X\to \R$.

Let $\{C^t\}_{t>0}$ be a parameterized family of sets in $\X$. Its inner and outer limit sets are defined, respectively,  by 
\begin{align*}
\liminf_{t\searrow 0} C^t&= \big\{x\in \X\, \big|\, \forall \; t_k \searrow 0 \;\exists \; x^{t_k}\to x \;\;\mbox{with}\;\; x^{t_k}\in C^{t_k}\; \; \mbox{for}\; \;  k\;\; \mbox{sufficiently large}\big\},\\
\limsup_{t\searrow 0} C^t&= \big\{x\in \X\, \big|\, \exists \; t_k \searrow 0 \;\exists\;   \; x^{t_k}\to x \;\;\mbox{with}\;\; x^{t_k}\in C^{t_k}\big\};
\end{align*}
see \cite[Definition~4.1]{rw}. 
The limit set of $\{C^t\}_{t>0}$ exists if  $\liminf_{t\searrow 0} C^t=\limsup_{t\searrow 0} C^t =:C$,  written as $C^t \to C$ when $t\searrow 0$. 
A sequence $\{f^k\b$ of functions $f^k:\X \to \oR$ is said to {epi-converge} to a function $f:\X \to \oR$ if we have $\epi f^k\to \epi f$ in $\X\times \R$ as $k\to \infty$; see \cite[Definition~7.1]{rw} for more details on the epi-convergence of a sequence of extended-real-valued functions. 
We denote by $f^k\xrightarrow{e} f$ the  epi-convergence of  $\{f^k\b$ to $f$.

Given a nonempty set $\Omega\subset\X $ with $\ox\in \Omega$, the tangent cone to $\Omega$ at $\ox$, denoted $T_\Omega(\ox)$,  is defined  by
$T_\Omega(\ox) = \limsup_{t\searrow 0} \frac{\Omega - \ox}{t}.$
The  regular/Fr\'{e}chet normal cone $\rN_\Omega(\ox)$ to $\Omega$ at $\ox$ is defined by
 $\rN_\Omega(\ox) = T_\Omega(\ox)^*$. 
For $x\notin \Omega$, we set $\rN_\Omega(x)=\emptyset$. 
The limiting/Mordukhovich normal cone $N_\Omega(\ox)$ to $\Omega$ at $\ox$ is the collection of all vectors $\ov\in \X $ for which there exist sequences  $\{x^k\b$ and  $\{v^k\b$ with $v^k\in \rN_\Omega( x^k)$ such that $(x^k,v^k)\to (\ox,\ov)$. 
When $\Omega$ is convex, both normal cones boil down to that of convex analysis.  
Given a function $f:\X  \to \oR$ and a point $\ox\in\X $ with $f(\ox)$ finite, the subderivative function $\d f(\ox)\colon\X \to\oR$ is defined by
\begin{equation*}
\d f(\ox)(w)=\liminf_{\substack{
   t\searrow 0 \\
  w'\to w
  }} {\frac{f(\ox+tw')-f(\ox)}{t}}, \quad w\in \X.
\end{equation*}
A vector $v\in \X $ is called a subgradient of $f$ at $\ox$ if $(v,-1)\in N_{\epi f}(\ox,f(\ox))$. 
The set of all subgradients of $f$ at $\ox$ is denoted by $\sub f(\ox)$. 
Replacing the limiting normal cone with $\rN_{\epi f}(\ox,f(\ox))$ in the definition of $\sub f(\ox)$ gives us $\what \sub f(\ox)$,  the regular subdifferential of $f$ at $\ox$.  
The function $f$ is called subdifferentially regular at $\ox$ provided that $N_{\epi f}(\ox, f(\ox)) = \rN_{\epi f}(\ox, f(\ox))$.
By definition, subdifferential regularity of $f$ at $\ox$ particularly ensures the coincidence $\sub f(\ox) = \what \sub f(\ox)$.
The critical cone of $f$ at $\ox$ for $\bar v$ with $\bar v\in   \sub f(\ox)$ is defined by 
\begin{equation*}
{K_f}(\ox,\bar v)=\big\{w\in \X \,\big|\,\la\bar v,w\ra=\d f(\ox)(w)\big\}.
\end{equation*}
When $f=\dd_\Omega$ for a nonempty subset $\Omega\subset\X$, the critical cone of $\dd_\Omega$ at $\ox$ for $\ov$ is denoted by $K_\Omega(\ox,\ov)$. 
In this case, one has $\d \dd_\Omega(\ox)=\dd_{T_\Omega(\ox)}$, the above definition of the critical cone of a function therefore boils down to  the well-known concept of critical cone to a set (see \cite[page~109]{DoR14}), 
namely $K_\Omega(\ox,\ov)=T_\Omega(\ox)\cap [\ov]^\perp$, where $[v]$ denotes the one-dimensional subspace $\{tv\, |\, t\in \R\}$.
If $f$ is subdifferentially regular and $\partial f(\ox) \neq \emptyset$, then $\d f(\ox)$ is the support function of $\partial f(\ox)$ (cf. \cite[Theorem~8.30]{rw}) and the critical cone $K_f(\ox, \ov)$ can be equivalently described by
\begin{equation}\label{crif}
K_f(\ox, \ov) = N_{\partial f(\ox)}(\ov).
\end{equation}
The {second subderivative} of $f$ at $\ox$ for $\ov$, denoted $\d^2 f(\bar x , \ov)$, is an extended-real-valued function defined  by 
\begin{equation*}
\d^2 f(\bar x , \ov)(w)= \liminf_{\substack{
   t\searrow 0 \\
  w'\to w
  }} \Delta_t^2 f(\ox , \ov)(w'),\;\; w\in \X,
\end{equation*}
where $\Delta_t^2 f(\ox , \ov),\, t>0,$ is the parametric  family of second-order difference quotients of $f$ at $\ox$ for $\ov\in \sub f(\ox)$ defined by 
$$
\Delta_t^2 f(\bar x , \ov)(w)=\frac{f(\ox+tw)-f(\ox)-t\langle \ov,\,w\rangle}{\frac {1}{2}t^2}
$$
for any $w\in \X$ and $t>0$. According to \cite[Definition~13.6]{rw},   $f$ is said to be {twice epi-differentiable} at $\bar x$ for $\ov$
if  the functions $ \Delta_t^2 f(\bar x , \ov)$ epi-converge to $  \d^2 f(\bar x,\ov)$ as $t\searrow 0$.  Further, we say that $f$ is {strictly} twice epi-differentiable at $\bar x$ for $\ov$ if the functions $  \Delta_t^2 f(  x , v)$ epi-converge  to some function as $t\searrow 0$ and $(x,v)\to (\ox,\ov)$ with $f(x)\to f(\ox)$ and $(x,v)\in \gph \sub f$.  
If this condition holds, the limit function is then the second subderivative $  \d^2 f(\bar x,\ov)$.

Consider a set-valued mapping $F:\X\tto \Y$. According to \cite[Definition~8.33]{rw}, the {graphical derivative} of $F$ at $\ox$ for $\oy$ with $(\ox,\oy) \in \gph F$ is the set-valued mapping $DF(\ox, \oy): \X\tto \Y$ defined via the tangent cone to $\gph F$ at $(\ox, \oy)$ by $\gph DF(\ox, \oy) = T_{\gph F}(\ox, \oy)$.  
Using the definition of the tangent cone, we can alternatively represent $\gph DF(\ox, \oy)$ in terms of graphical limit as
\begin{equation}\label{proto}
\gph DF(\ox, \oy) = \limsup_{t\searrow 0}\frac{\gph F-(\ox,\oy)}{t}.
\end{equation}
The set-valued mapping $F$ is said to be {proto-differentiable} at $\ox$ for $\oy$ if the outer graphical limit in \eqref{proto} is actually a full limit. 
When $F(\ox)$ is a singleton consisting of $\oy$ only, the notation $DF(\ox, \oy)$ is simplified to $DF(\ox)$. 
It is easy to see that for a single-valued mapping $F$, which is differentiable at $\ox$,  the graphical derivative  $DF(\ox)$ boils down to the Jacobian of $F$ at $\ox$, denoted by $\nabla F(\ox)$. 
Recall from \cite[Definition~9.53]{rw} that the strict graphical derivative of  a set-valued mapping $F$ at $\ox$ for $\oy$ with $(\ox,\oy)\in \gph F$, is  the set-valued mapping $\widetilde D F(\ox,\oy):\X\tto \Y$, defined  by 
\begin{equation}\label{sproto}
\gph \widetilde D F(\ox,\oy)=\limsup_{
  \substack{
   t\searrow 0 \\
  (x,y)\xrightarrow{ \gph F}(\ox,\oy)
  }} \frac{\gph F-(x,y)}{t}.
\end{equation}
We say that $F$ is {strictly} proto-differentiable at $\ox$ for $\oy$ if the outer graphical limit in \eqref{sproto} is attained as a full limit, or, equivalently, if its graph $\gph F$ is strictly smooth at $(\ox, \oy)$, cf. \cite[page~173]{r85}.
It is obvious from definitions that $\gph DF(\ox, \oy) \subset \gph \widetilde D F(\ox,\oy)$ in general.
Strict proto-differentiability of $F$ at $\ox$ for $\oy$ particularly implies that the latter  inclusion must hold as equality, that is $\widetilde D F(\ox, \oy)$ and $D F(\ox, \oy)$ coincide.

\section{Strict Twice Epi-Differentiability of $\C^2$-Partly Smooth Functions}\label{sect:sTED}
This section aims at establishing new second-order variational properties of $\C^2$-partly smooth functions as well as their applications in stability analysis of related variational systems.
We begin with reviewing  this notion and some of its well-known consequences obtained in \cite{Lew02, LeZ13}. 
Let $\M\subset \X$ be a $\C^2$-smooth manifold of codimension $m\leq \dim \X$ around the point $\ox \in \M$.
Equivalently, that means we can find an open neighborhood $\O\subset \X$ of $\ox$ and a  $\C^2$-smooth mapping $\Phi: \O \to \R^m$ such that 
\begin{equation}\label{mfold}
\M\cap \O = \big\{x\in \O\, \big|\, \Phi(x) = 0\big\} \quad \textrm{ and } \quad \nabla\Phi(\ox): \X \to \R^m\; \textrm{ is surjective},
\end{equation}
see, e.g., \cite[Proposition~5.16]{lee}.
We should stress that the surjectivity of $\nabla \Phi(\ox)$, assumed above, yields the validity of the same property of $\nabla \Phi(x): \X \to \R^m$  for all $x$ in an open neighborhood of $\ox$.
Thus, we can assume without loss of generality that $\nabla \Phi(x)$ is surjective for all $x\in \M\cap \O$.
It is well known (cf. \cite[Example~6.8]{rw}) that the tangent and normal cones to $\M$ at $x\in \M\cap \O$ are linear subspaces of $\X$, given by $T_\M(x) = \ker \nabla \Phi(x)$ and $N_\M(x) = \rge \nabla \Phi(x)^*$, respectively.
We should point out (cf. \cite[page~707]{Lew02}) that these two orthogonal complement linear subspaces are independent of the choice of $\Phi$, and that the critical cone to $\M$ at $x$ for $y$ enjoys the relationship 
\begin{equation}\label{criM}
K_\M(x, y) = T_\M(x)
\end{equation}
for all $x\in \M\cap \O$ and $y\in N_\M(x)$; see \eqref{crif} for the definition of the critical cone.

\begin{Definition}[$\C^2$-partial smoothness, cf. {\cite[Definition~3.2]{LeZ13}}]\label{def:ps}
A function $f: \X\to\oR$ is said to be $\C^2$-partly smooth at $\ox$ relative to a set $\M\subset \X$ containing $\ox$ if $\M$ is a $\C^2$-smooth manifold around $\ox$ and the following properties hold:
\begin{enumerate}[noitemsep,topsep=2pt]
\item (restricted smoothness) $f|_\M$ is $\C^2$-smooth around $\ox$, that is, there exists a representative function $\widehat f :\X\to \R$ which is $\C^2$-smooth around $\ox$ with $\widehat f |_\M=f|_\M$ locally around $\ox$;
\item (regularity) at every point $x\in \M$ close to $\ox$, the function $f$ is subdifferentially regular and has nonempty subdifferential $\partial f(x)$;
\item (normal sharpness) $N_\M(\ox) = \para\{\partial f(\ox)\}$;
\item (subgradient continuity) the subgradient mapping $\partial f$ is continuous at $\ox$ relative to $\M$ in the sense that the limit set $\lim_{x\in \M,\, x\to \ox}\sub f(x)$ exists and equals to $\sub f(\ox)$.
\end{enumerate}
The manifold $\M$ is often referred to as the {\em active manifold} at the point $\ox$, and also as the {\em optimal manifold} when $\ox$ is actually a minimizer of the function $f$.
\end{Definition}
Observe that  not only does a $\C^2$-partly smooth function  behave smoothly along some manifold according to  condition (a), but also satisfies regularity and normal sharpness properties in items (b)--(d).
These additional requirements create a strong bond between the function and its active manifold.
In fact, the active manifold associated with a $\C^2$-partly smooth function at a given point is locally unique in the presence of a suitable additional regularity, see the discussion prior to Proposition~\ref{prop:pscons}.
Note also that conditions (a), (b), and $\M$ being a $\C^2$-smooth manifold in the above definition are properties required in a neighborhood of the point $\ox$ in question.
On the other hand, the normal sharpness condition (c) was proved to be stable in the sense that if the function $f$ is $\C^2$-partly smooth at $\ox$ relative to $\M$, then the normal sharpness property in Definition~\ref{def:ps}(c) must be satisfied at all nearby points $x\in \M$ (cf. \cite[Proposition~2.10]{Lew02}).
Thus, if continuity of the subgradient mapping $x\in \M\mapsto \sub f(x)$ in (d) is maintained around $\ox$, then $\C^2$-partial smoothness of $f$ at $\ox$ relative to $\M$ yields the same property at all points $x\in \M$ nearby.
A localized version of that condition with respect to a dual variable $\ov \in \sub f(\ox)$,  which is absent from the definition above,  will be justified  in Proposition~\ref{prop:cont} to hold around $\ox$.
In the sequel, we refer to the properties described in Definition~\ref{def:ps} with condition (d) being replaced by 
\begin{enumerate}[noitemsep,topsep=2pt]
\item[(d')] there exists a neighborhood $V$ of $\ov$ such that $\lim_{x\in \M,\, x\to \ox}(V\cap \sub f(x)) = V\cap\sub f(\ox)$
\end{enumerate}
 as the {\em $\C^2$-partial smoothness of $f$ at $\ox$ for $\ov$}.
 Note also that this modified version of the $\C^2$-partial smoothness is stable (cf. Proposition~\ref{prop:cont}) and suffices for second-order variational analysis in this paper, since we are mainly concerned with pairs $(x,v)$ in  $\gph \sub f$.

The class of $\C^2$-partly smooth functions is known to encompass important functions that often appear in constrained and composite optimization problems, as mentioned in Section~\ref{sect:intro}.
We particularly refer the interested readers to \cite{Har06} for examples of $\C^2$-partly smooth functions and relations between partial smoothness structure and other smooth substructures of nonsmooth functions.
Below we briefly discuss the conditions outlined in Definition~\ref{def:ps} for some important classes of functions.

\begin{Example}[reliably $\C^2$-decomposable functions]\label{ex:decomp}
Following \cite{Sha03}, a function $f: \X\to \oR$ is called $\C^2$-decomposable at $\ox \in \X$ if $f(\ox)$ is finite and $f$ enjoys the composite representation
\begin{equation}\label{comp}
f(x) = f(\ox) + (\vt\circ\Phi)(x)\quad \textrm { for } x\in \O,
\end{equation}
where $\O\subset \X$ is an open neighborhood of $\ox$, $\vt: \Y\to \oR$ is proper, lsc, and sublinear, and $\Phi: \O\to \Y$ is $\C^2$-smooth with $\Phi(\ox) = 0$. 
Denote by $S$ the subspace of $\Y$ parallel to $\aff\{\partial \vt(0)\}$, namely $S = \para\{\partial \vt(0)\}$.
It was shown in \cite[pp.~683--684]{Sha03} that under the following nondegeneracy condition
\begin{equation}\label{nondeg}
S \cap \ker \nabla\Phi(\ox)^* = \{0\},
\end{equation}
the composite function $f$ in \eqref{comp} is $\C^2$-partly smooth at $\ox$ relative to the $\C^2$-smooth manifold $\M = \big\{ x\in \X\, \big|\, (P_S\circ\Phi)(x)=0\big\}$ with $\widehat f(x) = \la y, \Phi(x)\ra$, $x\in \X$, where $y\in \partial \vt (0)$ can be chosen arbitrarily. 
Moreover,  it is argued therein that subdifferential continuity in Definition~\ref{def:ps}(d) holds around the point $\ox$ relative to $\M$.
In \cite[Section~5]{HaS23}, a composite function, satisfying the local representation in \eqref{comp}  and the nondegeneracy condition in \eqref{nondeg}, is called reliably $\C^2$-decomposable at $\ox$.
Also, it was verified in \cite[Example~5.1]{HaS23} that this class encompasses a variety of important functions including polyhedral functions, certain composite functions, and indicator functions of important cones such as the second-order cone and the cone of semidefinite matrices.
\end{Example}

\begin{Example}[spectral functions]\label{ex:spectral}
Let $\X= \R^{n\times m}$, the space of $n\times m$ real matrices equipped with the trace inner product $\la X, Y\ra := \tr(X^\top Y)$ for all $X, Y \in \R^{n\times m}$.
Recall that singular values of a matrix $X\in \R^{n\times m}$ are the square roots of the eigenvalues of $X^\top X$.
Denote by $\sigma(X)$ the vector of $m$ singular values of a given matrix $X\in\R^{n\times m}$ with components $\sigma_1(X)\geq \sigma_2(X)\geq\cdots\geq\sigma_m(X)\geq 0$.
The correspondence $X\mapsto \sigma(X)$ defines a mapping from $\R^{n\times m}$ to $\R^m$.
In this example, we consider a spectral function $f: \R^{n\times m} \to \oR$ formulated as a composition by $f(X) = (\vt\circ\sigma)(X)$, where the outer function $\vt: \R^m \to \oR$ is lsc and absolutely permutation-invariant, meaning invariant under all signed permutations of coordinates.
Let $\overline{X}\in \R^{n\times m}$ and let $S\subset \R^m$ be a $\C^2$-smooth manifold around $\sigma(\overline{X})$ with $\delta_S$ being absolutely permutation-invariant, and $\M_{S} = \sigma^{-1}(S)$.
It was shown in \cite[Theorem~5.3]{DDL14} that $f$ is $\C^2$-partly smooth at $\overline{X}$ relative to $\M_S$ if and only if $\vt$ is $\C^2$-partly smooth at $\sigma(\overline{X})$ relative to $S$. 
\end{Example}

The next lemma presents a simple observation on normal cone to a convex set that will be employed to analyze next the class of convex piecewise linear-quadratic functions to see under which conditions such functions are $\C^2$-partly smooth.

\begin{Lemma}\label{nor_poly}
    Assume that $\Omega\subset \X$ is a convex set and $\ox\in \Omega$. Then 
    $\para\{\Omega\}^\perp\subset N_\Omega(\ox)$
    and equality holds if and only if $\ox \in \ri \Omega$.
\end{Lemma}

\begin{proof}
    It is clear that $\para\{\Omega\}^\perp\subset N_\Omega(x)$ for all $x\in \Omega$.
    Assume now that $\ox\in \ri\Omega$ and pick $\ov\in N_\Omega(\ox)$.
    There must exist $\epsilon>0$ such that $\B_\epsilon(\ox)\cap \aff\{\Omega\}\subset\Omega$.
    Observe that $$\la\ov, u\ra= \frac{1}{\epsilon}\la \ov, (\ox+\epsilon u)-\ox\ra\leq 0$$ for all $u\in\B\cap \para\{\Omega\}$, which in turn yields $\ov\in \para\{\Omega\}^\perp$. 
    Conversely, if $N_\Omega(\ox) = \para\{\Omega\}^\perp$, a subspace, then we infer via \cite[Proposition~2.51]{mn} that $\ox\in \ri\Omega$.
    The proof is then complete.
\end{proof}

\begin{Example}[convex piecewise linear-quadratic functions]\label{ex:cplq}
Let $\X=\R^n$ and consider a convex piecewise linear-quadratic (CPLQ) function $f$ defined on $\R^n$. 
Recall that a proper function $f: \R^n \to \oR$ is  piecewise linear-quadratic if $\dom f = \bigcup_{i\in {\cal I}} C_i$ with $C_i$ being a polyhedral convex set for any index $i$ from a finite set ${\cal I}$, relative to each of which $f$ takes the form
\begin{equation}\label{cplq}
f(x) = \sm \langle A_i x ,x \rangle + \langle a_i ,x \rangle + \alpha_i, \quad  x \in C_i,
\end{equation}
for some $n \times n$ symmetric matrix $A_i$, vector $a_i\in \R^n$, and scalar $\alpha_i\in \R$; cf. \cite[Definition~10.20]{rw}.
Since $f$ is convex, we know from \cite[Proposition~10.21]{rw} that $\dom f$ is polyhedral convex and $\partial f(x) \neq \emptyset$ for all $x\in \dom f$.
Thus, a CPLQ function always satisfies regularity condition in Definition~\ref{def:ps}(b).
Pick $\ox \in \dom f$ and define the set of active indices of the domain of $f$ at $\ox$ by
\begin{equation*}
{\cal I}(\ox) = \big\{i \in {\cal I}\, \big|\, \ox \in C_i\big\}.
\end{equation*}
We first analyze the simplest case when ${\cal I}(\ox)$ is a singleton, say ${\cal I}(\ox) = \{\oi\}$ for some $\oi \in {\cal I}$.
Since ${\cal I}(x)\subset {\cal I}(\ox)$ for all $x\in \dom f$ sufficiently close to $\ox$, there is a neighborhood of $\ox$ in which the function $f$ behaves as if it is comprised of a single quadratic piece, defined on a polyhedral convex set, namely 
\begin{equation*}
f(x) = \sm \langle A_\oi x ,x \rangle + \langle a_\oi ,x \rangle + \alpha_\oi + \delta_{C_\oi}(x), \quad x\in \R^n.
\end{equation*}
According to \cite[Theorem~18.2]{Roc70}, there exists a unique face of $C_\oi$ containing $\ox$ in its relative interior.
Let $\M_\ox$ be the relative interior of  such a face.
Clearly, $\M_\ox$ coincides with an affine subspace locally around $\ox$, and thus forms a $\C^2$-smooth manifold around that point. This tells us that $f$ restricted to $\M_\ox$ is  a quadratic function and thus is $\C^2$-smooth.
Moreover, observing that
\begin{equation*}
\partial f(x) = A_\oi x + a_\oi +N_{C_\oi}(x) = A_\oi x + a_\oi +N_{C_\oi}(\ox)
\end{equation*}
for all $x\in \M_\ox$, we get the fulfillment of the subdifferential continuitity in Definition~\ref{def:ps}(d).
It is also straightforward to see that $\para \{\partial f(\ox)\} = \spann \{N_{C_\oi}(\ox)\} = N_{\M_\ox}(\ox)$, meaning that the normal sharpness property in Definition~\ref{def:ps}(c) holds.
Thus, if ${\cal I}(\ox)$ is a singleton, we can conclude that $f$ is $\C^2$-partly smooth at $\ox$ relative to $\M_{\ox}$, where $\M_\ox$ is the maximal relatively open subset of $\dom f = C_\oi$ containing $\ox$, cf. \cite[Theorem~18.2]{Roc70}.

We now turn to the general case when ${\cal I}(\ox)$ may not be a singleton. 
Observe from the   analysis above for the case of a single linear-quadratic piece  that the manifold $\M_\ox$ consists of all points $x\in \R^n$ sharing not only a common set of active polyhedra $\{C_\oi\}$ but also the same normal cone $N_{C_\oi}(x)= N_{C_\oi}(\ox)$.
We aim to construct a set $\M_\ox$ enjoying this property   and to examine to what extent the CPLQ function $f$ is $\C^2$-partly smooth at $\ox$ relative to $\M_{\ox}$.

According to the analysis in the proof of \cite[Theorem~11.14(b)]{rw}, we  argue that such a set  $\M_\ox$ can be defined as follows.
Let ${\cal H}$ be a finite collection of closed half-spaces $H = \big\{x\in \R^n\, \big|\, \la b_H, x\ra \leq \beta_H\big\}$, for some $(b_H, \beta_H) \in (\R^n \setminus\{0\})\times \R$, such that 
(i) each of the polyhedral convex sets $\dom f$ and $C_i,\, i\in {\cal I},$ is the intersection of a subcollection of half-spaces in ${\cal H}$;
and (ii) for every $H\in {\cal H}$, the opposite closed half-space $H'= \big\{x\in \R^n\, \big|\, \la -b_H, x\ra \leq -\beta_H\big\}$ is also in ${\cal H}$.
Corresponding to this collection, define for each $\ox \in \dom f$ a subcollection ${\cal H}_\ox = \big\{H\in {\cal H}\, \big|\, \ox\in H\big\}$ and a relatively open set $\M_\ox = \big\{x\in \dom f\, \big|\, {\cal H}_x = {\cal H}_\ox\big\}$.
Thus, we have ${\cal I}(x)={\cal I}(\ox)$ for all $x\in \M_\ox$, and  $N_{C_i}(x)=N_{C_i}(\ox)$ for any $x\in \M_\ox$ and any $i\in {\cal I}(\ox)$. Note also that $\M_\ox$ has the alternative representation 
\begin{equation}\label{mfold1}
\M_\ox = \ri \bigcap_{H\in {\cal H}_\ox}H = \Big(\bigcap_{H\in {\cal H}_\ox^*} H\cap H'\Big)\bigcap\Big(\bigcap_{H\in {\cal H}_\ox \setminus {\cal H}_\ox^*} \inte H\Big),
\end{equation}
where ${\cal H}_\ox^* = \big\{H\in {\cal H}_\ox\, \big|\, H'\in {\cal H}_\ox\big\}$.
This clearly shows that in some neighborhood of $\ox$, $\M_\ox$ coincides with the affine subspace $\bigcap_{H\in {\cal H}_\ox^*} (H\cap H')$.
Hence, $\M_\ox$ forms a $\C^2$-smooth manifold around $\ox$ and satisfies
\begin{equation*}
N_{\M_\ox}(x) = \spann\big\{b_H\,|\, H\in {\cal H}_\ox^*\big\}\quad\textrm{ for all }\; x \in \M_\ox.
\end{equation*}
Fix an index $\oi\in {\cal I}(\ox)$.
Observing that $\M_\ox\cap C_\oi \neq\emptyset$, we get $\M_\ox\subset C_\oi$ and conclude that $f$ restricted to $\M_\ox$ coincides with $\sm \langle A_\oi x ,x \rangle + \langle a_\oi ,x \rangle + \alpha_\oi$ and thus is $\C^2$-smooth.
Recall also from \cite[page~487]{rw} that 
\begin{equation}\label{subcplq}
\partial f(x) = \bigcap_{i\in {\cal I}(\ox)}\big(A_ix+a_i+ N_{C_i}(\ox)\big)
\subset A_\oi x+a_\oi+ N_{\M_\ox}(\ox)
\end{equation}
for all $x\in \M_\ox$, where the last inclusion is due to \cite[Proposition~2.2]{Lew02}.

{\bf Claim.} {\em The CPLQ function $f$ in \eqref{cplq} is $\C^2$-partly smooth at $\ox$ relative to the affine manifold $\M_\ox$ in \eqref{mfold1} if and only if the following conditions are satisfied:
\begin{equation}\label{cplq1}
\begin{cases}
\spann \{N_{C_i}(\ox)\} = N_{\M_\ox}(\ox) \quad \textrm{ for all } \; i\in {\cal I}(\ox),\; \textrm{ and }\\
\bigcap_{i\in {\cal I}(\ox)} (A_i\ox + a_i +\ri N_{C_i}(\ox))\neq \emptyset.
\end{cases}
\end{equation}
}

To verify the ``if" part, assume  that the conditions in \eqref{cplq1} hold. 
Pick $v\in \ri \sub f(\ox)$ and conclude from \eqref{subcplq} and \cite[Proposition~2.42]{rw} that $v\in A_i\ox + a_i +\ri N_{C_i}(\ox)$ for all $i\in {\cal I}(\ox)$. Employing Lemma~\ref{nor_poly} and the first condition in \eqref{cplq1} implies for any  $i \in {\cal I}(\ox)$ that 
$$N_{N_{C_i}(\ox)}(v-A_i\ox-a_i) = \spann \{N_{C_i}(\ox)\}^\perp= T_{\M_\ox}(\ox).$$ 
We then get from  \cite[Corollary~23.8.1]{Roc70} and \eqref{subcplq} that
\begin{equation}\label{cplq2}
N_{\sub f(\ox)}(v) = \sum_{i\in {\cal I}(\ox)} N_{N_{C_i}(\ox)}(v-A_i\ox-a_i) = \sum_{i\in {\cal I}(\ox)} T_{\M_\ox}(\ox) = T_{\M_\ox}(\ox).
\end{equation}
Recalling that $v\in \ri \partial f(\ox)$, we infer from Lemma~\ref{nor_poly} that $N_{\sub f(\ox)}(v) = \para\{\partial f(\ox)\}^\perp$, which, together with \eqref{cplq2}, implies the normal sharpness condition $\para\{\sub f(\ox)\}=N_{\M_\ox}(\ox)$.
It follows from  \cite[Proposition~2.42]{rw} and the second condition in \eqref{cplq1} that the sets
$A_i\ox +a_i+ N_{C_i}(\ox)$ and $\bigcap_{k\in {\cal I}(\ox)\setminus\{i\}} (A_k\ox + a_k + N_{C_k}(\ox))$ have relative interior points in common for all $i\in \I(\ox)$.
These amount to none of the sets $A_i\ox + a_i +N_{C_i}(\ox)$ being separated properly from the intersection $\bigcap_{k\in {\cal I}(\ox)\setminus\{i\}} (A_k\ox + a_k + N_{C_k}(\ox))$ of the others, cf. \cite[Theorem~11.3]{Roc70}.
Employing \cite[Exercise~4.33]{rw}, we can then deduce from the   representation of $\partial f(x)$ in \eqref{subcplq} that $\partial f$ is continuous at $\ox$ relative to $\M_\ox$. Thus, $f$ is $\C^2$-partly smooth at $\ox$ with respect to $\M_\ox$, defined in \eqref{mfold1}.

We now turn to the ``only if" part.
First observe from \eqref{subcplq} and the inclusions $\M_\ox\subset C_i, \, i\in \I(\ox),$ that 
\begin{equation}\label{eq90}
\para\{\partial f(\ox)\} \subset \spann \{N_{C_i}(\ox)\}\subset N_{\M_\ox}(\ox), \quad i\in \I(\ox).
\end{equation}
The normal sharpness condition  $\para\{\sub f(\ox)\}=N_{\M_\ox}(\ox)$ then yields  the first condition in \eqref{cplq1} for all $i\in {\cal I}(\ox)$.
We prove next that $v\in A_i\ox + a_i +\ri N_{C_i}(\ox)$ for all $i\in {\cal I}(\ox)$.
Assume that $v - A_\oi\ox -a_\oi \notin\ri N_{C_\oi}(\ox)$ for some $\oi\in {\cal I}(\ox)$.
We then get from $\spann \{N_{C_\oi}(\ox)\}^\perp = T_{\M_\ox}(\ox)$ and Lemma~\ref{nor_poly} that $T_{\M_\ox}(\ox)\subsetneq N_{N_{C_\oi}(\ox)}(v-A_\oi\ox-a_\oi)$.
The first equality in \eqref{cplq2}, which is always valid due to     \eqref{subcplq}, therefore leads to $T_{\M_\ox}(\ox)\subsetneq N_{\sub f(\ox)}(v) = N_{\M_\ox}(\ox)^\perp$, a contradiction.
Thus, we arrive at $v\in \bigcap_{i\in I(\ox)} \big(A_i\ox + a_i + \ri N_{C_i}(\ox)\big)$ and complete verification of the claim.

Observe that conditions required in \eqref{cplq1} automatically hold when ${\cal I}(\ox)$ is a singleton, as analyzed above.
However, \eqref{cplq1} should not be taken for granted in general.
A simple counterexample is given by $f(x) = \tfrac{1}{2}(\max\{x, 0\})^2, \, x\in \R$. 
It can be seen that $f$ is $\C^1$-smooth everywhere but not $\C^2$-partly smooth at $\ox=0$ relative to $\M_\ox = \{0\}$, since the normal sharpness property does not hold: $\para\{\nabla f(\ox)\} = \{0\}$ while $N_{\M_\ox}(\ox) = \R$.
\end{Example}

It is noticed, in our forthcoming analysis, that subdifferential regularity in Definition~\ref{def:ps}(b) should be replaced with a stronger condition called {\em prox-regularity}.
In fact, this was well observed in the literature on partial smoothness.
For example, in \cite[Section~7]{Lew02}, Lewis constructed an example of an everywhere subdifferentially regular function that is $\C^2$-partly smooth  relative to two distinct manifolds.
It was then pointed out in \cite[Corollary~4.12]{LeZ13} that such a concern is eliminated for prox-regular functions.
More precisely, the additional prox-regularity ensures local uniqueness of the corresponding active manifolds, and hence, it does have appealing effects on $\C^2$-partly smooth functions, see, e.g., \cite[Sections~4 and 5]{LeZ13}.
Prox-regularity was therefore a key ingredient in exploiting the partial smoothness structure, see, e.g., \cite{DDL14, DHM06, DIL16}.
Also, prox-regularity and subdifferential continuity of the function in question were essential assumptions throughout the analysis carried out in \cite{HaS23} on strict proto-differentiability of subgradient mappings.
Among appealing implications of those two properties, we should mention the equivalence between (strict) twice epi-differentiability and (strict) subgradient proto-differentiability; cf. \cite[Theorem~13.40]{rw} and \cite[Corollary~6.2(b)]{pr}.
These will be extensively exploited in our establishment of strict twice epi-differentiability of a $\C^2$-partly smooth function and its consequences in this section.
We now recall the notion of such regularity.
According to \cite[Definition~13.27]{rw}, a function  $f\colon\X\to\oR$ is {prox-regular} at $\ox$ for $\ov$ if $f$ is finite at $\ox$ and locally lower semicontinuous (lsc)  around $\ox$ with $\ov\in\sub f(\ox)$, and there exist $\epsilon>0$ and $\rho\geq 0$ such that
\begin{equation}\label{prox}
\begin{cases}
f(u)\ge f(x')+\la v',u-x'\ra-\frac{\rho}{2}\|u-x'\|^2 \quad\textrm{ for all } \; u\in \B_\epsilon(\ox)\\
\textrm{whenever } \; (x', v')\in \B_{\epsilon}(\ox,\ov)\cap\gph\sub f\; \textrm{ and }\; f(x') \leq f(\ox) +\epsilon.
\end{cases}
\end{equation}
When this property holds for all $v\in \sub f(\ox)$, the function $f$ is called prox-regular at $\ox$.
The function $f$ is said to be {subdifferentially continuous} at $\ox$ for $\ov$ if the convergence $(x^k,v^k)\to(\ox,\ov)$ with $v^k\in\sub f(x^k)$ yields $f(x^k)\to f(\ox)$ as $k\to\infty$.
The localization in function value in \eqref{prox} can be omitted if $f$ is subdifferentially continuous at $\ox$ for $\ov$.

\begin{Proposition}\label{local_prox}
Assume that $f\colon\X\to\oR$ is prox-regular and subdifferentially continuous at $\ox$ for $\ov$. Then, there exists $\ve>0$ such that 
$f$ is prox-regular and subdifferentially continuous at $x$ for $v$
for any $(x, v)\in \B_\ve(\ox,\ov)\cap \gph \sub f$.
\end{Proposition}

\begin{proof} 
Pick the constants $\epsilon$ and $\rho$ for which \eqref{prox} holds. 
Since $f$ is subdifferentially continuous at $\ox$ for $\ov$,
we find $\ve\in (0, \epsilon/2)$ such that $|f(x) - f(\ox)|< \epsilon$ for any $(x,v)\in \B_\ve(\ox,\ov)\cap \gph \sub f$. 
Pick any $(x,v)\in \B_{\ve}(\ox,\ov)\cap \gph \sub f$.
It is easy to see that the prox-regularity inequality in \eqref{prox} with the same constant $\rho$ is satisfied for all $u\in \B_{\epsilon/2}(x)$ and all $(x', v')\in \B_{\epsilon/2}(x,v)\cap\gph\sub f$.
This proves prox-regularity of $f$ at $x$ for $v$. 
Shrinking $\ve$ if necessary, we deduce from \cite[Proposition~2.3]{pr} that $f(x')\to f(x)$ whenever $(x',v')\to (x,v)$ with $(x',v')\in \B_{\ve}(\ox,\ov)\cap \gph \sub f$, which proves subdifferential continuity of $f$ at $x$ for $v$ and completes the proof.
\end{proof}


The following result summarizes important properties of $\C^2$-partly smooth functions obtained in \cite{Lew02, LeZ13}.
While the local representation of graphs of $\C^2$-partly smooth functions in \eqref{ps2} below is the cornerstone in our justification of strict twice epi-differentiability of such functions, stability of the relative interior condition plays a key role in extending the latter point-based property to all points nearby.
The latter stability can be derived from \cite[Lemma~20]{DHM06}, however, the equivalence between our statement and the latter lemma is probably not apparent.
We instead supply a direct proof for stability of the relative interior condition based on arguments in the proof of \cite[Proposition~4.3]{LeZ13}.

\begin{Proposition}\label{prop:pscons}
Let $f:\X\to \oR$ be a $\C^2$-partly smooth function at $\ox$ relative to a $\C^2$-smooth manifold $\M$. 
Let $\widehat f : \X\to \R$ be any $\C^2$-smooth representative of $f$ around $\ox$ relative to $\M$. 
The following assertions hold.
\begin{enumerate}[noitemsep,topsep=2pt]
\item We have $\nabla \widehat f  (\ox)\in \aff\{\partial f(\ox)\}$.

\item Suppose, moreover, that $f$ is prox-regular and subdifferentially continuous at $\ox$ for some $\ov \in \ri \partial f(\ox)$. Then there exists $\varepsilon > 0$ such that
\begin{equation}\label{ps2}
\B_\varepsilon(\ox, \ov)\cap \gph \partial f = \B_\varepsilon(\ox, \ov)\cap \gph (\nabla \widehat f  + N_\M)
\end{equation}
and $v\in \ri \partial f(x)$ for all $(x, v) \in \B_\varepsilon(\ox, \ov)\cap \gph \partial f$.
\end{enumerate}
\end{Proposition}

\begin{proof}
First, assertion (a) follows immediately from \cite[Proposition~2.4]{Lew02}.
We now turn to (b).
Since the function $f$ is $\C^2$-partly smooth at $\ox$ relative to $\M$ and both prox-regular and subdifferentially continuous at $\ox$ for $\ov \in \ri \partial f(\ox)$, the local representation of $\gph \partial f$ in \eqref{ps2}  results from \cite[Corollary~5.2]{LeZ13}.

It remains to show that the relative interior condition is stable in some neighborhood of $(\ox, \ov)$ relative to $\gph \partial f$. 
The proof of this part mimics that of \cite[Proposition~4.3]{LeZ13}. 
Assume the contrary that there exists a sequence $\{(\xk, \vk)\}_{k\in \N}\subset \gph \partial f$ with $\xk \to \ox$ and $\vk \to \ov$ as $k\to\infty$ and $\vk \notin \ri \partial f(\xk)$ for all $k$. 
The local representation in  \eqref{ps2} then particularly tells us that eventually $x^k \in \M$.
Recall that $f$ is subdifferentially regular at all point $x\in \M$ close to $\ox$.
We can assume that $\partial f(\xk)$ is closed and convex for all sufficiently large $k\in \N$.
Due to stability of the normal sharpness property as noted above, we can also assume that $\para\{\partial f(\xk)\} = N_\M(\xk)$ for all such $k$. 
Employing now the separation theorem for the nonempty closed convex set $\partial f(\xk)$ and a point $\vk$ from its relative boundary (cf. \cite[Exercise~2.45(e)]{rw}),  we find $\eta^k \in \para\{\partial f(\xk)\} = N_\M(\xk)$ with $\|\eta^k\|=1$ such that $\la \eta^k, \vk\ra \geq \la \eta^k, u\ra$ for all $u\in \partial f(\xk)$. 
Let $\bar\eta$ be any accumulation point of $\{\eta^k\}_{k\in \N}$.
Clearly, $\bar\eta \neq 0$.
Moreover, we conclude via robustness of the limiting normal cone that $\bar\eta \in N_\M(\ox) = \para\{\partial f(\ox)\}$.
Recall that $\ov\in \ri \partial f(\ox)$.
There must exist $\epsilon > 0$ such that $\B_\epsilon(\ov)\cap \aff\{\partial f(\ox)\} \subset \partial f(\ox)$.
Pick an arbitrary $u\in \B_\epsilon(\ov)\cap \aff\{\partial f(\ox)\}$.
Since $u\in \partial f(\ox)$, we can find $\uk \to u$ with $\uk \in \partial f(\xk)$ for all $k$ sufficiently large due to subdifferential continuity of $\partial f$ at $\ox$ relative to $\M$ and to the facts that $\xk \to \ox$ and $\xk\in\M$ as $k\to\infty$. 
Noting that $\la \eta^k, \vk\ra \geq \la \eta^k, \uk\ra$ for all sufficiently large $k$, we arrive at $\la \bar\eta, \ov\ra \geq \la \bar\eta, u\ra$.
Since $u\in \partial f(\ox)$ is chosen arbitrarily from $\B_\epsilon(\ov)\cap \aff\{\partial f(\ox)\}$, the latter implies that $\la \bar \eta, u\ra \leq 0$ for all $u \in \epsilon\B\cap \para\{\partial f (\ox)\}$,
a contradiction with $0\neq \bar\eta \in \para\{\partial f(\ox)\}$. The proof is then complete.
\end{proof}

\begin{Remark}\label{sub_cont} We should point out that the subdifferential continuity
assumption in Proposition~\ref{prop:pscons} can be dropped from this result with no harm. Indeed, one can conclude from \cite[Proposition~10.12]{DrL12} that the crucial representation in \eqref{ps2} does hold without the latter condition. Having this result in our disposal, we can drop the subdifferential continuity
assumption from all the results in this paper.  However, we proceed with assuming it to avoid more complication in our presentation.
\end{Remark}

\begin{Proposition}[continuity of subgradient mapping]\label{prop:cont}
Let $f:\X\to \oR$ be a $\C^2$-partly smooth function at $\ox$ relative to a $\C^2$-smooth manifold $\M$ such that $f$ is prox-regular and subdifferentially continuous at $\ox$ for some $\ov \in \ri \partial f(\ox)$. 
Then there exists a neighborhood $V$ of $\ov$ such that the subgradient mapping $x\in \M\mapsto V\cap \sub f(x)$ is continuous around $\ox$.
\end{Proposition}

\begin{proof}
According to Proposition~\ref{prop:pscons}, the local representation of $\gph \sub f$ in \eqref{ps2} with $\B_\ve(\ox, \ov)$ replaced with $\B_\ve(\ox)\times \B_\ve(\ov)$ is valid for some $\ve>0$ and some $\C^2$-smooth function $\widehat f$.
Set $V:=\B_{\ve/2}(\ov)$ and pick $x\in  \B_{\ve/2}(\ox)\cap \M$, $v\in V\cap \sub f(x)$, and $\xk \to x$ with $\xk\in \M$.
We now construct a sequence $v^k$ converging to $v$ with $v^k \in \sub f(\xk)$ for all large $k$.
It follows from \eqref{ps2} that $v-\nabla \widehat f(x)\in N_\M(x)$.
Due to surjectivity of $\nabla\Phi(\ox)$ and smoothness of $\Phi$ around $\ox$, the normal cone mapping $x\mapsto N_\M(x) = \rge\nabla\Phi(x)^*$ is continuous around $\ox$.
Choosing a smaller $\ve$ if necessary, we can assume that $N_\M$ is continuous at $x$ and then, for the given sequence $\{\xk\}_{k\in \N}$, we find $y^k \to v-\nabla \widehat f(x)$ with $y^k\in N_\M(\xk)$ for all $k$.
Let $v^k:= \nabla \widehat f(\xk) +y^k$ and observe that $v^k \to v$.
We then have $\xk \in \B_\ve(\ox)$ and $v^k \in \B_\ve(\ov)$, for all $k$ sufficiently large, and thus arrive at $(\xk, v^k)\in (
\B_\ve(\ox)\times \B_\ve(\ov))\cap \gph(\nabla \widehat f+N_\M)$ for such $k$.
Again, it follows from \eqref{ps2} that $v^k\in \sub f(\xk)$ for all $k$ sufficiently large.
This confirms continuity of the localization of subgradient mapping $\sub f$ around $\ox$ relative to $\M$ and completes the proof.
\end{proof}

Below, we present  our main result in this section in which strict twice epi-differentiability of $\C^2$-partly smooth functions is justified. Taking a $\C^2$-partly smooth function at $\ox$ relative to $\M$, denoted $f$, and assuming $\widehat f : \X\to \R$ as a $\C^2$-smooth representative of $f$ around $\ox$ relative to $\M$, we define the Lagrangian function $L:\X\times\R^m\to \R$ associated with $f$ by 
\begin{equation}\label{lagf}
L(u, y):= \widehat f (u) +\la y, \Phi(u)\ra, \quad  (u, y)\in \X\times \R^m,  
\end{equation}
where the mapping $\Phi$ comes from \eqref{mfold}. 

\begin{Theorem}[strict twice epi-differentiability]\label{thm:ps}
Let $\M$ be a $\C^2$-smooth manifold around $\ox$ with the local representation \eqref{mfold} and $f:\X\to \oR$ a $\C^2$-partly smooth function at $\ox$ relative to $\M$ such that $f$ is prox-regular and subdifferentially continuous at $\ox$ for some $\ov \in \ri \partial f(\ox)$. 
Let $\widehat f : \X\to \R$ be any $\C^2$-smooth representative of $f$ around $\ox$ relative to $\M$. 
Then, for all $(x, v)\in \gph \partial f$ sufficiently close to $(\ox, \ov)$,
the following properties hold:
\begin{enumerate}[noitemsep,topsep=2pt]
\item
the subgradient mapping $\partial f$ is strictly proto-differentiable at $x$ for $v$ and  its graphical derivative $D (\partial f)(x, v)$ can be calculated by
\begin{equation}\label{ps1}
D (\partial f)(x, v) (w)= \begin{cases}
\nabla_{xx}^2L(x, \mu)(w)+N_\M(x)\quad&\textrm{if }\; w \in T_\M(x),\\
\emptyset&\textrm{otherwise},
\end{cases}
\end{equation}
where $\mu\in \R^m$ is the unique vector satisfying $\nabla \Phi(x)^*\mu = v - \nabla \widehat f (x)$;
\item the function $f$ is strictly twice epi-differentiable at $x$ for $v$ and  its  second subderivative can be calculated by
\begin{equation}\label{ps3}
\d^2f(x, v)(w)=\nabla_{xx}^2L(x, \mu)(w, w) +\delta_{T_\M(x)}(w), \quad w\in \X.
\end{equation}
\end{enumerate}
\end{Theorem}

\begin{proof}
We first verify the claimed assertions in (a) and (b) for $(\ox, \ov)$ with $\omu$ being the unique vector satisfying $\nabla \Phi(\ox)^*\omu = \ov - \nabla \widehat f (\ox)$.
To this end, we begin by showing  that $\widehat f +\delta_\M$ is strictly twice epi-differentiable at $\ox$ for $\ov$. 
Take $(x,\eta)\in \gph N_\M$ with $x$ sufficiently close to $\ox$. 
According to \cite[Theorem~6.2]{mms}, $\delta_\M$ is   twice epi-differentible at $x $  for   $\eta$, 
and its second subderivative at $x $  for   $\eta$ is given by
\begin{equation}\label{co3.1}
\d^2\delta_\M(x, \eta)(w) = \la \nu, \nabla \Phi^2(\ox)(w, w)\ra +\delta_{K_\M(x,\, \eta)}(w), \quad w\in \X,
\end{equation}
where $\nu \in \R^m$ is the unique vector satisfying $\nabla \Phi(x)^*\nu = \eta$. 
Since $\widehat f $ is $\C^2$-smooth,  \cite[Exercise~13.18]{rw} tells us that $\widehat f +\delta_\M$ maintains twice epi-differentiability at such $x$ for $\nabla \widehat f (x)+\eta$ and 
\begin{equation}\label{ps5}
\d^2(\widehat f +\delta_\M)(x, \nabla \widehat f (x) + \eta)(w) = \nabla^2\widehat f (x)(w, w) +\d^2\delta_\M(x, \eta)(w), \quad w\in \X.
\end{equation}
Note further from Proposition~\ref{prop:pscons}(a) and the normal sharpness in Definition~\ref{def:ps} that $\ov - \nabla \widehat f (\ox) \in \para\{\partial f(\ox)\} = N_\M(\ox)$.
Since $N_\M(\ox)$ is obviously a linear subspace, we get  $\ov - \nabla \widehat f(\ox) \in \ri \sub \delta_\M(\ox)$ and conclude from \cite[Theorem~5.8]{HaS23} that $\delta_\M$ is strictly twice epi-differentiable at $\ox$ for $\ov - \nabla \widehat f (\ox)$.
By \cite[Proposition~5.6]{HaS23} and the aforementioned twice epi-differentiability of $\delta_\M$, we obtain for any  pair $(x, \eta)\to (\ox, \ov - \nabla \widehat f (\ox))$ with $\eta\in N_\M(x)$ that
\begin{equation*}
\d^2\delta_\M(x,\eta) \xrightarrow{e}\d^2\delta_\M(\ox, \ov - \nabla \widehat f (\ox)).
\end{equation*}
Employing now the sum rule for epi-convergence in \cite[Theorem~7.46(b)]{rw}, we deduce from the latter and \eqref{ps5} that
\begin{equation*}
\d^2(\widehat f +\delta_\M)(x, \nabla \widehat f (x) +\eta) \xrightarrow{e}\nabla^2\widehat f (\ox)(\cdot, \cdot) +\d^2\delta_\M(\ox, \ov - \nabla \widehat f (\ox))
\end{equation*}
as $(x, \eta)\to (\ox, \ov - \nabla \widehat f (\ox))$ with $\eta\in N_\M(x)$. 
We then conclude via \cite[Proposition~5.6]{HaS23} that $\widehat f +\delta_\M$ is strictly twice epi-differentiable at $\ox$ for $\ov$ and that 
\begin{equation*}
\d^2(\widehat f +\delta_\M)(\ox, \ov)(w) = \nabla^2\widehat f (\ox)(w, w) +\d^2
\delta_\M(\ox, \ov - \nabla \widehat f (\ox))(w), \quad w\in \X.
\end{equation*}
Recall that $\omu$ is the unique vector satisfying the equation $\nabla \Phi(\ox)^*\bar\mu = \ov - \nabla \widehat f (\ox)$.
In view of \eqref{criM} and \eqref{co3.1}, we get from the above formula that
\begin{align*}
\d^2(\widehat f +\delta_\M)(\ox, \ov)(w)
&= \nabla^2\widehat f (\ox)(w, w) +\la \omu, \nabla^2\Phi(\ox)(w,w)\ra +\delta_{T_\M(\ox)}(w)\\
&=\nabla^2_{xx} L(\ox, \omu)(w, w) +\delta_{T_\M(\ox)}(w),\quad w\in \X.
\end{align*}
By \cite[Example~10.24(f)]{rw},  $\widehat f +\delta_\M$ is strongly amenable at $\ox$ in the sense of  \cite[Definition~10.23(b)]{rw}. Thus, by \cite[Theorem~13.32]{rw},   it is  prox-regular and subdifferentially continuous at any $x\in \M$ sufficiently close to $\ox$.
Appealing now to \cite[Corollary~4.3]{pr2} and \cite[Theorem~13.40]{rw}, we deduce that the subgradient mapping $\partial (\widehat f +\delta_\M) = \nabla \widehat f  +N_\M$ is strictly proto-differentiable at $\ox$ for $\ov$ and that 
\begin{align*}
D (\nabla \widehat f  +N_\M)(\ox, \ov)(w) &= \partial \big(\tfrac{1}{2}\nabla^2_{xx} L(\ox, \omu)(\cdot, \cdot) +\delta_{T_\M(\ox)}\big)(w)\\
& = \nabla^2_{xx} L(\ox, \omu)(w) + N_{T_\M(\ox)}(w), \quad w\in \X.
\end{align*} 
According to Proposition~\ref{prop:pscons}(a), $\gph \sub f$ and $\gph \sub(\widehat f +\delta_\M) $ locally coincide around $(\ox,\ov)$. This
yields strict proto-differentiability of $\partial f$ at $\ox$ for $\ov$. 
It also confirms that 
$$
D(\sub f)(\ox,\ov)= D (\nabla \widehat f  +N_\M)(\ox, \ov)
$$
and hence proves the claimed formula  in \eqref{ps1}. 
Moreover, we infer via \cite[Corollary~4.3]{rw} that $f$ is strictly twice epi-differentiable  at $\ox$ for $\ov$.
Remembering from \cite[Theorem~13.40]{rw} that 
$$\partial \big(\tfrac{1}{2}\d^2f(\ox, \ov)\big) = D(\partial f)(\ox, \ov),$$
and setting $\ph(w):= \nabla_{xx}^2L(\ox, \omu)(w, w) +\delta_{T_\M(\ox)}(w)$ for any $w\in \X$, we deduce from \eqref{ps1} that 
$\partial \big(\d^2f(\ox, \ov)\big)(w)=\sub \ph(w)$
for any $w\in \X$. 
By \cite[Proposition~13.49]{rw}, we find $r>0$ such that the mapping $w\mapsto \d^2f(\ox, \ov)(w)+ r\|w\|^2$ is a proper lsc convex function.
Choosing a bigger $r$ if necessary, we can assume without loss of generality that $w\mapsto \ph(w)+ r\|w\|^2$ is also a proper lsc convex function. 
Appealing to \cite[Theorem~12.25]{rw}, which is an integration of the subdifferential for convex functions, tells us that there is a constant $c$ for which we have $\d^2f(\ox, \ov)(w)+ r\|w\|^2=\ph(w)+ r\|w\|^2 +c$ for any $w\in \X$. 
Since $\d^2f(\ox, \ov)(0)=0$, we get $c=0$, 
which proves the formula \eqref{ps3} for $(\ox, \ov)$.

We now argue that assertions (a) and (b)  are also valid for any $(x, v)\in \gph \partial f$ sufficiently close to  $(\ox, \ov)$.
It follows from the discussion after Definition~\ref{def:ps} and Proposition~\ref{prop:cont} that $f$ is $\C^2$-partly smooth at $x$ for $v\in \partial f(x)$ for all $(x, v)$ sufficiently close to $(\ox, \ov)$.
Also, Proposition~\ref{local_prox} tells us that  prox-regularity and subdifferential continuity of $f$ at $x$ for $v$ are preserved whenever $(x,v)\in \gph\partial f$ and sufficiently close to $(\ox, \ov)$.
For any such a pair $(x, v)$, we can deduce, without loss of generality, via Proposition~\ref{prop:pscons}(b) that $v\in \ri\partial f(x)$.
Using an argument for any such a pair $(x,v)$ similar to that for $(\ox, \ov)$ proves (a) and (b) for $(x,v)$ and hence completes the proof.
\end{proof}

Strict twice epi-differentiability in a neighborhood relative to the graph of subgradient mapping was first characterized for polyhedral functions in \cite[Theorem~4.1]{HJS22} via the same relative interior condition in Theorem~\ref{thm:ps}.
A similar result was then achieved for composite functions whose outer functions are polyhedral convex in \cite[Theorem~3.9]{HaS22} under a second-order qualification condition (SOQC).
We should stress that these structures, namely, polyhedral function and $\C^2$-smooth mapping composed with polyhedral function satisfying SOQC, are stable in the sense that if a given function has such a structure at a given point, it maintains the same structure at all points nearby.
In \cite[Theorem~5.8]{HaS23}, by assuming the reliable $\C^2$-decomposability of the function in a neighborhood of the point in question, the characterization was extended to the class of reliably $\C^2$-decomposable functions considered in Example~\ref{ex:decomp}.
However, Example~\ref{ex:decomp} and Theorem~\ref{thm:ps} together tell us that reliable $\C^2$-decomposability at the point in question solely can guarantee the aforementioned characterization.
This discloses the power of robustness of $\C^2$-partial smoothness.
We should remark, however, that in Theorem~\ref{thm:ps} the relative interior condition serves as a sufficient condition, not as a characterization, for strict twice epi-diffrentiability.
In fact, the essence of $\C^2$-partial smoothness is rooted in the local representation of $\gph \sub f$ in \eqref{ps2} which is available only for multipliers taken from the relative interior of the subdifferential, to the best of our knowledge.

\begin{Remark}
Note that it is possible to get another proof of Theorem~\ref{thm:ps} using \cite[Theorem~4.4]{pr2} and \cite[Theorem~28]{DHM06}. Indeed, it was shown in the former that for prox-regular functions, strict twice epi-differentiability amounts to continuous differentiability of proximal mappings under an extra condition that the function has a global minimum. The latter restrictive condition was weakened  recently by the authors in \cite[Theorem~4.7]{HaS23}, where it was argued that prox-boundedness suffices to achieve this equivalence. On the other hand, it was shown in \cite[Theorem~28]{DHM06}  for prox-regular and prox-bounded functions that are $\C^2$-partly smooth that the proximal mapping is continuously differentiable. Combining these results tells us that, providing the same relative interior condition imposed in Theorem~\ref{thm:ps} is satisfied, any prox-regular, prox-bounded, and $\C^2$-partly smooth function is strictly twice epi-differentiable. Theorem~\ref{thm:ps} provides new information, however. First, it shows that prox-boundedness is not required. Second, it proves strict twice epi-differentiability of a prox-regular and $\C^2$-partly smooth function in a neighborhood of a given point. Lastly, it presents a simple formula for the second subderivative of such functions that has been reported before and is a byproduct of our direct proof for this result.  
\end{Remark}

\begin{Remark}
Assume that $f: \X\to \oR$ is reliably $\C^2$-decomposable at $\ox \in \X$ with representation \eqref{comp} and that $\ov \in \ri\partial f(\ox)$.
Pick an arbitrary $y \in \partial \vt (0)$.
Example~\ref{ex:decomp} asserts that $f$ is $\C^2$-partly smooth at $\ox$ relative to the $\C^2$-smooth manifold $\M = \big\{ x\in \X\, \big|\, \Phi(x)\in S^\perp\big\}$, where $S = \para\{\sub \vt (0)\}$, and $\widehat f (x) = \la y, \Phi(x)\ra$ is a $\C^2$-smooth representative of $f$ around $\ox$ relative to $\M$.
It can be shown that the composite function \eqref{comp} satisfying \eqref{nondeg} is prox-regular and subdifferentially continuous at $\ox$ for $\ov$, since it is strongly amenable at $\ox$ in the sense of  \cite[Definition~10.23(b)]{rw}.
Let $\omu\in \Y$ be any solution to $\nabla(P_S\circ\Phi)(\ox)^*\mu = \ov - \nabla\Phi(\ox)^*y$.
Recalling that the projection $P_S$ is linear and self-adjoint, we get from the latter that
$$
\nabla \Phi(\ox)^*P_S(\bar\mu)=\ov -\nabla\widehat f(\ox).
$$
It then follows from  \eqref{nondeg} that $\eta := P_S(\omu)$ is the unique vector in $S$ satisfying $\nabla\Phi(\ox)^*\eta = \ov - \nabla\widehat f(\ox)$.
We now can conclude via Theorem~\ref{thm:ps}  that $f$ is strictly twice epi-differentiable at $\ox$ for $\ov$ and that 
\begin{equation}\label{d2comp}
\d^2f(\ox, \ov)(w) = 
\nabla_{xx}^2\la y + P_S(\omu), \Phi\ra(\ox)(w, w) +\delta_{T_\M(\ox)}(w), \quad w\in \X.
\end{equation}
We next justify that $\oy := y + P_S(\omu) \in \partial \vt(0)$ and $\nabla\Phi(\ox)^*\oy = \ov$.
Indeed, the latter follows from definition of $\omu$.
To prove the former, let $\hat y$ be the unique vector from $\partial \vt (0)$ satisfying $\nabla\Phi(\ox)^*\hat y = \ov$, thanks to \eqref{nondeg}.
Since $\ov \in\ri\partial f(\ox)$, we infer via \cite[Proposition~5.3(b)]{HaS23} that $\hat y \in \ri \partial \vt (0)$.
For any $t\in [0, 1)$, we have $(1-t)\hat y + t y \in \ri \partial \vt (0)$ which results in 
$$
(1-t)\hat y +t\oy = (1-t)\hat y + t y + t P_S(\omu) \in \partial \vt (0)
$$
for all $t\geq0$ sufficiently small due to $S = \para\{\partial \vt (0)\}$. 
Because  $\nabla \Phi(\ox)^* ((1-t)\hat y + t\oy) = \ov$, 
we then conclude from the uniqueness of $\hat y$ that $\oy = \hat y$.
Recalling the composite form \eqref{comp} of $f$, we have $K_f(\ox, \ov) = \big\{w\in \X\, \big|\, \nabla \Phi(\ox) w \in K_\vt(0, \oy)\big\}$.
Using sublinearity of $\vt$, we have $\vt(u) =\d\vt(0)(u)$ for any $u\in \dom \vt$, which coupled with \cite[Theorem~8.24 and Corollary~8.25]{rw} leads us to 
\begin{equation*}
K_\vt(0, \oy) = \big\{u\in \Y\, \big|\, \vt(u)  =  \la \oy, u\ra\big\} = \big\{u\in \Y\, \big|\, \oy \in \argmax_{y\in \partial \vt (0)} \la u, y\ra \big\}.
\end{equation*}
Since $\oy\in \ri \partial \vt (0)$, the linear function $\la u, \cdot\ra$ attains its maximum over $\partial \vt (0)$ at $\oy$ if and only if $u \in (\aff\{\partial \vt (0)\})^\bot = S^\bot$.
Thus, we have
\begin{equation*}
K_f(\ox, \ov) = \big\{w\in \X\, \big|\, \nabla \Phi(\ox) w \in S^\bot\big\} = T_\M(\ox),
\end{equation*}
where the last equality follows from \eqref{nondeg}.
Combining this, \eqref{d2comp}, and the fact that $\oy = y + P_S(\omu)$ is the unique vector from $\partial \vt (0)$ satisfying $\nabla\Phi(\ox)^*\oy = \ov$ implies that  Theorem~\ref{thm:ps} covers the implication (c) $\Longrightarrow$ (a), and also the implication (c) $\Longrightarrow$ (b), in \cite[Theorem~5.8]{HaS23}.
We should stress here that in order to apply Theorem~\ref{thm:ps}, one only needs to justify the reliable $\C^2$-decomposability of the function $f$ at the point $\ox$. On the contrast,  \cite[Theorem~5.8]{HaS23} demands the latter property in a neighborhood of $\ox$.
Thus, Theorem~\ref{thm:ps} actually improves those implications.
Particularly, we infer from the latter theorem that a proper, lsc, and sublinear function $\vt: \Y\to \oR$ is strictly twice epi-differentiable at $z\in \Y$ for $y\in \sub \vt(z)$ for all $(z, y)\in \gph \sub \vt$ close to $(0, \oy)$, for any $\oy \in \ri \sub \vt(0)$, which is an improvement of  \cite[Theorem~5.7]{HaS23}. 
\end{Remark}


Stability of the strict twice-epidifferentiability of a $\C^2$-partly smooth function, established in Theorem~\ref{thm:ps}, makes the class of such functions a proper subset of the class of strictly twice epi-differentiable functions as demonstrated in the next  example.

\begin{Example}[failure of strict twice epi-diffrentiability in a neighborhood]
To demonstrate that  the class of strictly twice epi-differentiable functions is strictly larger than that of $\C^2$-partly smooth functions, 
consider  the function $f(x) = \int_0^xg(t)dt, \, x\in \R$, where $g: \R\to\R$ is  a continuous piecewise affine function given by
\begin{equation*}
g(x) = \begin{cases}
0&\mbox{ if }\; x \leq 0,\\
2^{2i} + 3\cdot2^i(x-2^i)\quad&\mbox{ if }\; x \in  [2^i, 2^{i+1}), \; i\in \mathbb{Z}.
\end{cases}
\end{equation*}
Observe that $f$ is $\C^1$-smooth with $\nabla f (x) = g(x)$ for all $x\in \R$. 
Moreover, it can be seen that $g$ is monotonically increasing over $\R$, which ensures that $f$ is convex.
Note also that the affine pieces comprising the graph of  $g$ have increasing slopes of $0$ and $3\cdot2^i,\, i \in \mathbb{Z}$.
Assume that $t$ and $u$ are close to $0$ and that $t>u$ with $t\in [2^{-i}, 2^{-i+1})$ for some $i\in \N$. If $u\leq 0< t$, we arrive at
$$g(t)-g(u) = g(t) = 3\cdot 2^{-i}t - 2^{-2i+1}\leq 3\cdot 2^{-i}t\leq 3\cdot 2^{-i}(t-u).$$
If $t> u\ge 0$, then $u \in [2^{-j}, 2^{-j+1})$ for some $j\geq i$. Assume that $j>i$ and consider the lines
$$y=g_i(x) := 3\cdot 2^{-i}x-2^{-2i+1} \quad\textrm{ and }\quad y=g_j(x) := 3\cdot 2^{-j}x-2^{-2j+1},$$
which intersect each other at the point whose $x$-coordinate is $x_0=\frac{2}{3}2^{-i}+\frac{1}{3}2^{-j+1}$.
Since $u\leq x_0\leq t$, we have $g(u)=g_j(u)\geq g_i(u)$ and
\begin{align*}
   g(t)-g(u) \leq g_i(t) - g_i(u)= 3\cdot 2^{-i}(t-u). 
\end{align*}
Clearly, the latter estimate is also valid for $j=i$. 
Combining these tells us that 
\begin{equation*}
0\leq \frac{g(t) - g(u)}{t-u} \leq\frac{3\cdot2^{-i}(t-u)}{t-u}\leq 3t\to 0\quad \mbox{ as }\; t, u \searrow 0.
\end{equation*}
Finally, we have $g(t) - g(u) = 0$  if  $u<t\le 0$.
Thus, $g$ is strictly differentiable at $\ox = 0$ with $\nabla g(0) = 0$.
Note, however, that $g$ is not differentiable at any $x= 2^i, i\in \mathbb{Z},$ since those points correspond to kinks on $\gph g$.
In view of \cite[Proposition~3.1(b)]{r85} and \cite[Theorem~4.3]{pr2}, we conclude that the convex function $f$ is strictly twice epi-differentiable at $\ox=0$ for $\nabla f(\ox) = 0$ while  not satisfying such a property along the sequence $x^k:= 2^{-k},\, k\in \N,$ which converges to $0$.
Theorem~\ref{thm:ps} then tells us that $f$ is not  $\C^2$-partly smooth at $0$.
\end{Example}


As Theorem~\ref{thm:ps} shows, under the assumptions therein a $\C^2$-partly smooth function is always strictly twice epi-differentiable around the point under consideration.
One may wonder whether strict twice epi-differentiability in a neighborhood actually amounts to $\C^2$-partial smoothness.
The example below rules out this possibility.   

\begin{Example}[failure of $\C^2$-partial smoothness]
Define the function $f: \R^2\to \R$  by $f(x_1, x_2) = |x_1^3x_2^2|$, which is taken from \cite[Example~4.11]{pr3}. 
Since $f$ can be expressed equivalently as $\max\{x_1^3x_2^2, -x_1^3x_2^2\}$, it is prox-regular and subdifferentially continuous on $\R^2$.
Moreover, $f$ is $\C^1$-smooth on $\R^2$   and $\C^2$-smooth around any point $(x_1, x_2)$ with $ x_1\neq 0$.
Its gradient and Hessian, when exists, can be calculated as 
\begin{equation*}
\nabla f(x_1, x_2) = 
    \sign(x_1) ( 3x_1^2x_2^2,2x_1^3x_2 ) 
\end{equation*}
and
\begin{equation*}
    \nabla^2f(x_1, x_2) = \sign(x_1)\begin{bmatrix}
     6x_1x_2^2 & 6x_1^2x_2\\
     6x_1^2x_2 & 2x_1^3
\end{bmatrix} \quad\textrm{ if }\; x_1\neq 0,
\end{equation*}
where $\sign(x_1)$ signifies the sign of $x_1$: $\sign(x_1) = 1$ if $x_1>0$, $\sign(x_1)=-1$ if $x_1<0$, and $\sign(x_1)=0$ otherwise.
For any $w=(w_1,w_2)\in \R^2$, it is not hard to see that 
\begin{equation*}
\d^2f(x_1, x_2)(w) = \begin{cases}
0&\textrm{if }\; (x_1, x_2) \in \{0\}\times \R,\\
\nabla^2f(x_1, x_2)(w,w)\quad&\textrm{otherwise}.
\end{cases}
\end{equation*}
Observe from these formulas that $\d^2f(x_1', x_2')\to\d^2f(x_1, x_2)$ uniformly on bounded sets whenever $(x_1', x_2')\to (x_1, x_2)$.
It then follows from \cite[Proposition~7.15]{rw} that the latter uniform convergence implies $\d^2f(x_1', x_2')\xrightarrow{e}\d^2f(x_1, x_2)$ when $(x_1', x_2')\to (x_1, x_2)$, or, equivalently, $f$ is strictly twice epi-differentiable at  any $(x_1, x_2)\in\R^2$.
However, we see that $f$ is not $\C^2$-partly smooth at $\ox:=(0,0)$. 
Indeed,  if that was the case, we would conclude from the smoothness of $f$ at $\ox$ and the normal sharpness property of $f$ at this point that   $N_\M(\ox)=\{0\}$, where $\M$ is the associated manifold to $f$ from Definition~\ref{def:ps}. 
By \cite[Exercise~6.19]{rw}, we arrive at $\ox\in \inte \M$, which tells us that $f$ must be $\C^2$-smooth around $\ox$, a contradiction to the calculation above. 
This confirms that $f$ is not $\C^2$-partly smooth at $\ox$.
\end{Example}

Recall that for a set-valued mapping $F:\X\tto\Y$ with $(\ox, \oy) \in \gph F$, the coderivative mapping of $F$ at $\ox$ for $\oy$, denoted $D^*F(\ox, \ov)$, is defined via the normal cone to $\gph F$ at $(\ox, \oy)$ by
\begin{equation}\label{cod}
\eta\in D^*F(\ox, \oy)(w) \iff (\eta, -w)\in N_{\gph F}(\ox, \oy).
\end{equation}

\begin{Corollary}\label{cod_sgraph}
Let $f:\X\to \oR$ be a $\C^2$-partly smooth function at $\ox$ relative to a $\C^2$-smooth manifold $\M$ such that $f$ is prox-regular and subdifferentially continuous at $\ox$ for  $\ov$ with  $\ov\in \ri \partial f(\ox)$. Then, for any $(x,v)\in \gph \sub f$ sufficiently close to $(\ox, \ov)$, we have
\begin{equation}\label{cod1}
D(\sub f)(x, v) = \widetilde D(\sub f)(x, v)= D^*(\sub f)(x, v).
\end{equation}
\end{Corollary}

\begin{proof} According to Theorem~\ref{thm:ps}, $f$ is strictly proto-differentiable at $x$ for $v$ for any $(x,v)\in  \gph \sub f$ sufficiently close to $(\ox,\ov)$, which in turn justifies the first claimed equality. The second one then results from the recent observation in \cite[Theorem~3.9]{HaS23}.
\end{proof}

We should add here that  the local representation of $\gph \sub f$ in \eqref{ps2} for a $\C^2$-partly smooth function $f$  was exploited in \cite[Theorem~5.3 and Corollary~5.4]{LeZ13} to derive an exact formula for 
$D^*(\sub f)(\ox, \ov)$ in terms of the covariant Hessian $\nabla^2_\M f_\ov(\ox)$ of the tilted function $f_\ov:= f - \la \ov, \cdot\ra$.
Note that the bilinear form $\nabla^2_{xx}L(\ox, \olm)$ in \eqref{ps1} restricted to $T_\M(\ox) \times T_\M(\ox)$ gives $\nabla^2_\M f_\ov(\ox)$; see \cite[Definition~2.11]{LeZ13} and the discussion afterwards.
Corollary~\ref{cod_sgraph} not only provides a new proof for this result but also goes one step further and calculates the strict graphical derivative of the subgradient mapping for this class of functions.

Recall that a function $f:\X\to \oR$ enjoys the quadratic growth condition at 
$\ox$ if $f(\ox)$ is finite and there exists $\ell>0$ such that 
\begin{equation}\label{qgc}
f(x)\ge f(\ox)+\frac{\ell}{2}\|x-\ox\|^2
\end{equation}
for all $x$ close to $\ox$. 
In the corollary below, we glean a characterization for the quadratic growth of $\C^2$-partly smooth functions from our calculation of the second subderivative of those functions in the proof of Theorem~\ref{thm:ps}. 
The same result was previously observed in \cite[Propostion~4.13]{LeZ13} using a different approach.

\begin{Corollary}\label{sub_cpar} Under the hypothesis of Theorem~{\rm\ref{thm:ps}}, there exists $\ve>0$ such that for any $(x,v)\in \B_\ve(\ox,\ov)\cap \gph \sub f$ we have 
\begin{equation}\label{qg1}
\d^2f(x, v) = \d^2(\widehat f +\delta_\M)(x, v).
\end{equation}
Consequently, if $\ov=0$ then $f$ enjoys the quadratic growth condition at 
$\ox$ if and only if $\hat f$ enjoys the same property at $\ox$ with respect to  $\M$, namely \eqref{qgc} holds  with $f$ replaced with $\hat f$ for all $x\in \M$ close to $\ox$. 
\end{Corollary}
\begin{proof} The first claim about the second subderivative of $f$ was indeed proved in the proof of Theorem~\ref{thm:ps}. The second claim results from \eqref{qg1} for $(x, v)=(\ox, 0)$ and the fact that the validity of quadratic growth condition can be fully characterized by positiveness of the second subderivative; see \cite[Theorem~13.24(c)]{rw}. 
\end{proof}

\begin{Remark}
It is worth pointing out that in \cite[Propostion~4.13]{LeZ13} the subdifferential continuity was not assumed. Taking into account the discussion in Remark~\ref{sub_cont}, it is possible to drop the latter condition from Corollary~\ref{sub_cpar} as well. 
\end{Remark}

The rest of this section is devoted to some applications of the established strict twice epi-differentiability of $\C^2$-partly smooth functions. 
We begin with regularity properties of the solution mapping to the generalized equation
\begin{equation}\label{cGE}
\op\in G(x):=\opsi (x) +\sub f(x),
\end{equation}
where $f: \X\to \oR$ is a  $\C^2$-partly smooth function, $\opsi: \X\to \X$ is $\C^1$-smooth around the point under consideration, and $\op\in\X$.
We are mainly interested in the solution mapping to \eqref{cGE}, that is the 
mapping $p\mapsto G^{-1}(p)$. 

Recall that a set-valued mapping $F:\X \tto \Y$ is said to be {metrically regular} at $\ox$ for $\oy\in F(\ox)$ if there exist a positive constant $\kappa$  and  neighborhoods $U$ of $\ox$ and $V$ of $\oy$ such that the estimate
$\dist \big(x, F^{-1}(y)\big)\leq \kappa\, \dist\big(y, F(x)\big)$
holds for all $(x, y)\in U\times V$.
The mapping $F$ is said to be {strongly metrically regular} at $\ox$ for $\oy$ if its inverse $F^{-1}$ admits a Lipschitz continuous single-valued localization around $\oy$ for $\ox$, which means that there exist neighborhoods $U$ of $\ox$ and $V$ of $\oy$ such that the mapping $y\mapsto F^{-1}(y)\cap U$ is single-valued and Lipschitz continuous on $V$. 
According to \cite[Proposition~3G.1]{DoR14}, strong metric regularity of $F$ at $\ox$ for $\oy$ amounts to $F$ being metrically regular at $\ox$ for $\oy$ and the inverse mapping $F^{-1}$ admitting a single-valued localization around $\oy$ for $\ox$.
These regularity properties
of a set-valued mapping enjoy full characterizations via the coderivative in \eqref{cod} (cf. \cite[Theorem~9.43]{rw}), and the strict graphical derivative in \eqref{sproto} (cf. \cite[Theorem~9.54(b)]{rw}),
respectively.
We recently  demonstrated in \cite[Proposition~4.1]{HaS23} that these regularity properties are equivalent for the mapping $G$ in \eqref{cGE} provided that $\sub f$ is strictly proto-differentiable. Moreover,  we presented several simple characterizations for those equivalent properties of $G$. 
In the presence of strict proto-differentiability of $\sub f$ in a neighborhood relative to $\gph \sub f$, established for $\C^2$-partly smooth functions in Theorem~\ref{thm:ps}, we provide below an adaption of the aforementioned results for the mapping $G$ from \eqref{cGE}.

\begin{Proposition}\label{prop:GE}
Let $\ox$ be a solution to the generalized equation in \eqref{cGE} for $\op\in \X$. Assume that $\opsi$ is $\C^1$-smooth around $\ox$ and $f$ is $\C^2$-partly smooth at $\ox$ relative to a $\C^2$-smooth manifold $\M$ with local representation \eqref{mfold}.
Assume in addition that $\ov:=\op-\opsi(\ox) \in \ri \partial f(\ox)$, and that $f$ is  prox-regular and subdifferentially continuous at $\ox$ for $\ov$.
Then the following properties are equivalent.
\begin{itemize}[noitemsep,topsep=2pt]
\item [\rm{(a)}] The mapping $G$ from \eqref{cGE} is metrically regular at $\ox$ for $\op$.

\item [\rm{(b)}] The mapping $G$ from \eqref{cGE} is strongly metrically regular at $\ox$ for $\op$.

\item [\rm{(c)}] The inverse mapping $G^{-1}$ has a Lipschitz continuous single-valued localization $s$ around $\op$ for $\ox$.

\item [\rm{(d)}] The condition 
\begin{equation}\label{ge1}
\big\{ w\in T_\M(\ox)\, \big|\, \big(\nabla\opsi(\ox)+\nabla^2_{xx}L(\ox, \omu)\big)^*(w) \in N_\M(\ox)\big\} =\{0\}
\end{equation}
is satisfied, where $L$ is taken from \eqref{lagf} and $\omu$ is the unique solution to $\nabla\Phi(\ox)^*\omu = \op- (\opsi(\ox)+\nabla \widehat f  (x))$
with $\widehat f : \X\to \R$ being any $\C^2$-smooth representative of $f$ around $\ox$ relative to $\M$.

\item [\rm{(e)}] The composite mapping $P_{T_\M(\ox)}\circ(\nabla\opsi(\ox)+\nabla^2_{xx}L(\ox, \omu))\vert_{T_\M(\ox)}$, where 
$(\nabla\opsi(\ox)+\nabla^2_{xx}L(\ox, \omu))\vert_{T_\M(\ox)}$ stands for the restriction of the linear mapping $\nabla\opsi(\ox)+\nabla^2_{xx}L(\ox, \omu)$ to $T_\M(\ox)$,  is an one-to-one linear mapping from $T_\M(\ox)$ onto $T_\M(\ox)$, where $L$ and $\omu$ are taken from {\rm (d)}.
\end{itemize} 

Moreover, if any of the above equivalent properties holds, then the mapping $s$ in {\rm (c)} is $\M$-valued, namely $s(p)\in \M$ for any $p$ in a neighborhood of $\op$, and $\C^1$-smooth around $\op$ with
\begin{equation}\label{Jloc1}
\nabla s(p) = DG(x, p)^{-1} = \big(P_{T_\M(x)}\circ(\nabla\opsi(x)+\nabla^2_{xx}L(x, \mu))\vert_{T_\M(x)}\big)^{-1}\circ P_{T_\M(x)}.
\end{equation}
where $x = s(p)$ and $\mu\in \R^m$ is the unique vector satisfying $\nabla\Phi(x)^*\mu =  p-(\opsi(x) + \nabla \widehat f (x))$.
\end{Proposition}

\begin{proof}
Since $\ov \in \ri \partial f(\ox)$, we infer from Theorem~\ref{thm:ps} that $\sub f$ is strictly proto-differentiable at $x$ for $v$ for all $(x, v)\in \gph\sub f$ sufficiently close to $(\ox, \ov)$ and its graphical derivative $D(\sub f)(x, v)$ is calculated by \eqref{ps1}.
Thus, the equivalence among assertions (a)--(d) is an immediate consequence of \cite[Proposition~4.1, {\rm(a)}\;$\Longleftrightarrow$\;{\rm(b)}]{HaS23} and \cite[Theorems~4.2~ and~4.3]{HaS23}.
It remains to verify that {\rm (d)} and {\rm (e)} are equivalent.
If (d) holds, taking the orthogonal complements from both sides of \eqref{ge1} and using \cite[Corollary~11.25]{rw} tell us  that \eqref{ge1} is equivalent to
\begin{equation}\label{MRcri2}
N_\M(\ox) + (\nabla\opsi(\ox) + \nabla^2_{xx}L(\ox, \omu))(T_\M(\ox)) = \X.
\end{equation}
This, combined with a similar argument as  the proof of \cite[Theorem~4.2, {\rm(a)}\;$\Longleftrightarrow$\;{\rm(b)}]{HaS23}, leads us to the desirable claim in {\rm (e)}.
Conversely, assuming {\rm (e)}, we arrive at 
$$\dim \big((\nabla\opsi(\ox) + \nabla^2_{xx}L(\ox, \omu))(T_\M(\ox))\big) = \dim T_\M(\ox)$$
and
$$(\nabla\opsi(\ox) + \nabla^2_{xx}L(\ox, \omu))(T_\M(\ox))\cap N_\M(\ox)=\{0\},$$
which together give us \eqref{MRcri2}, and hence \eqref{ge1}. This completes the proof of the equivalence of {\rm (d)} and {\rm (e)}.

Assume now that one of the   equivalent properties in {\rm (a)}--{\rm (e)} is satisfied.
By (c), we find neighborhoods $U$ and $V$ of $\ox$ and $\op$, respectively, such that the mapping $s(p):= U\cap G^{-1}(p), \, p\in V,$ is single-valued and Lipschitz continuous on $V$.
It then holds that $(s(p), p - \opsi(s(p))) \in (U\times \widetilde V)\cap\gph \sub f$, where $\widetilde V$ is an open set containing $(I_\X - \opsi\circ s)(V)$, where $I_\X$ stands for the identity mapping from $\X$ onto $\X$.
Obviously, $\widetilde V$ is a neighborhood of $\ov = \op-\opsi(\ox)$.
Shrinking $U$ and $V$, if necessary, we can conclude from \eqref{ps2} that $s(p)\in \M$ for all $p\in V$.
Recall that  $\sub f$ is strictly proto-differentiable at $x$ for $v$ for all $(x, v) \in \gph \sub f$ near $(\ox, \ov)$.
We then conclude via \cite[Theorem~4.3]{HaS23} that $s: V\to U\cap \M$ is $\C^1$-smooth around $\op$.
We proceed by justifying \eqref{Jloc1}.
Observe that metric regularity of $G$ holds at $x$ for $p$ for all $(x, p)\in \gph G$ sufficiently close to $(\ox, \op)$.
By shrinking $U$ and $V$, we can assume the latter property  holds for all $(x, p)\in (U \times V) \cap\gph G$ and also differentiability of $s$ at $p$.
Pick now $p\in V$ and $u\in \X$.
Since $s$ is differentiable at $p$, we get that $w:=\nabla s(p)(u) = Ds(p)(u) = DG^{-1}(p, x)(u)$, where $x:=s(p)\in U$.
The latter verifies the first equality in \eqref{Jloc1} and also gives us $u \in DG(x, p)(w)$.
Employing the sum rule for graphical derivative from \cite[Exercise~10.43(a)]{rw}, we have $u\in \nabla \opsi(x)w + D(\sub f)(x, p-\nabla \opsi(x))(w)$.
This, together with \eqref{ps1}, yields
$w\in T_\M(x)$ and $u \in  \nabla \opsi (x)w+\nabla^2_{xx}L(x, \mu)w + N_\M(x)$, where $\mu$ is the unique vector satisfying $\nabla\Phi(x)^*\mu = p - (\nabla\opsi(x)+\nabla \widehat f  (x))$.
The equivalence {\rm (a)} $\Longleftrightarrow$ {\rm (e)}, together with $G$ being metrically  regular at $x$ for $p$, tells us that the mapping $P_{T_\M(x)}\circ(\nabla\opsi(x)+\nabla^2_{xx}L(x, \mu))\vert_{T_\M(x)}: T_\M(x) \to T_\M(x)$ is one-to-one.
Projecting $u$ onto the subspace $T_\M(x)$, we have
\begin{equation*}\label{mr2}
P_{T_\M(x)}(u) = P_{T_\M(x)}((\nabla\opsi(x)+\nabla^2_{xx}L(x, \mu))(w)),
\end{equation*}
Since $w\in T_\M(x)$, we get from the latter that
\begin{equation*}
w 
= \big(P_{T_\M(x)}\circ(\nabla\opsi(x)+\nabla^2_{xx} L(x, \mu))\vert_{T_\M(x)}\big)^{-1}(P_{T_\M(x)}(u)),
\end{equation*}
which confirms \eqref{Jloc1} and hence completes the proof.
\end{proof}

Note that a characterization of continuous differentiability of the solution mapping $s$, defined in Proposition~\ref{prop:GE}, which resembles the one in \eqref{Jloc1}, was recently established  in \cite[Theorem~2.7]{DDJ23} via a different approach. However, the equivalence between metric regularity and strong metric regularity was not obtained therein.

As an application of the Proposition~\ref{prop:GE} above, we present next $\C^1$-smoothness of the proximal mapping of $\C^2$-partly smooth functions. 
Recall that proximal mapping of a function $f:\X\to\oR$ for a parameter $r>0$, denoted $\prox_{rf}$, is defined by
\begin{equation*}
\prox_{rf}(x) = \argmin_{w\in \X}\big\{f(w)+\tfrac{1}{2r}\|w-x\|^2\big\}, \quad x\in \X.
\end{equation*}
Moreover, we demonstrate that the proximal mapping is able to eventually identify the active manifold of a $\C^2$-partly smooth function.
The same results were established for projection mapping onto a prox-regular and $\C^2$-partly smooth set -- a set is called $\C^2$-partly smooth if  its indicator function  is $\C^2$-partly smooth-- in \cite[Theorem~3.3]{HaL04}, and for the proximal mapping of a prox-bounded, prox-regular, and $\C^2$-partly smooth function in \cite[Theorem~28]{DIL16}.
Proofs of  these results  rely mainly on \cite[Theorem~5.7]{Lew02},  which establishes stability of strong critical points of parametric $\C^2$-partly smooth functions.
Below, we  derive these results as an immediate consequence of  Proposition~\ref{prop:GE}. 
To begin with, recall that a function $f$ is said to be {prox-bounded} if there exists a real number $\alpha$ such that the function $f+\alpha \|\cdot\|^2$ is bounded from below on $\X$; see \cite[Exercise~1.24]{rw}.
It is worth mentioning that while we will assume the subdifferential continuity in our result below, which was not assumed in 
\cite{DIL16}, it is possible to drop using the discussion in Remark~\ref{sub_cont}.

\begin{Corollary}\label{cor:prox}
Let $f:\X\to \oR$ be a $\C^2$-partly smooth function at $\ox$ relative to a $\C^2$-smooth manifold $\M$ with local representation \eqref{mfold} and let $\ov  \in \ri \partial f(\ox)$.
Assume that $f$ is prox-regular and subdifferentially continuous at $\ox$ for $\ov$ and that $f$ is prox-bounded.
Then for any $r>0$ sufficiently small, the proximal mapping $\prox_{rf}$ is $\M$-valued and $\C^1$-smooth around $\ox + r \ov$ with
\begin{equation}\label{Jprox}
\nabla(\prox_{rf})(z) = \big(P_{T_\M(x)}\circ (I_\X+r\nabla^2_{xx}L(x, \mu))\vert_{T_\M(x)}\big)^{-1}\circ P_{T_\M(x)},
\end{equation}
where $L$ is taken from \eqref{lagf}, $x = \prox_{rf}(z)$, and $\mu\in \R^m$ is the unique vector satisfying $\nabla_xL(x, \mu)= (z - x)/r$.
\end{Corollary}

\begin{proof}
Let $\epsilon>0$ and $\rho\geq 0$ be constants for which \eqref{prox} holds. 
Fixing $r \in (0, 1/\rho)$, with convention $1/0 = \infty$, we can deduce from \cite[Theorem~13.37]{rw} that $\prox_{rf} = (I_\X+rT_\epsilon)^{-1}$,
where $T_\epsilon$ is a localization of $\sub f$ around $(\ox, \ov)$ whose graph coincides with $\gph\sub f$ in $\B_\epsilon(\ox, \ov)$,   is moreover single-valued and Lipschitz continuous around $\ox+ r \ov$.
Considering $\opsi = I_\X$ and $\op = \ox+r\ov$, we then infer from Proposition~\ref{prop:GE} that $\prox_{rf}$, a Lipschitz continuous single-valued localization of the inverse mapping of $G = I_\X+r\sub f$ around $\ox + r\ov$ for $\ox$, is $\M$-valued and $\C^1$-smooth around that point.
The formula for the Jacobian of $\prox_{rf}$ at $z$ sufficiently close to $\ox + r\ov$ in \eqref{Jprox} follows immediately from that in \eqref{Jloc1}, which completes the proof.
\end{proof}


We close this section with an extension of Proposition~\ref{prop:GE} into stability analysis of the generalized equation 
\begin{equation}\label{GE}
0 \in \psi(p, x) +\sub f(x)
\end{equation}
associated with the subgardient mapping of a $\C^2$-partly smooth function. 
We will deal with the mapping $\psi: \cP\times \X\to \X$ that concerns with a general perturbation represented by parameter $p\in \cP$, where $\cP\subset \P$ is an open subset of a normed space $\P$.
Define the solution mapping $S: \P\tto\X$ by
\begin{equation}\label{sol}
S(p) := \big\{x\in \X\, \big|\, 0\in\psi(p, x) + \sub f (x)\big\}
\end{equation}
for $p\in \cP$ and $S(p) = \emptyset$ otherwise.
We are interested in conditions ensuring local single-valuedness and certain differential stability properties of $S$ around a given parameter $\op\in \cP$.
With characterizations for strong metric regularity in Proposition~\ref{prop:GE}, we are in a position to deliver the last result of this section establishing semidifferentiability and differentiability of the solution mapping in \eqref{sol}.
This will be the driving force behind our investigation in the next section.
Recall that for a set-valued mapping $F: \X \tto \Y$ with $(\ox, \oy) \in \gph F$, if the limit
\begin{equation*}
\lim_{\substack{
t\searrow 0\\w'\to w}}
\frac{F(\ox+tw')-\oy}{t}
\end{equation*}
exists for all $w\in \X$, then $F$ is said to be semidifferentiable at $\ox$ for $\oy$ and the limit, which equals to  $DF(\ox, \oy)(w)$, is called the semiderivative of $F$ at $\ox$ for $\oy$ and $w$.
For a single-valued mapping $F$, we simply call the limit, denoted by $DF(\ox)(w)$, the semiderivative of $F$ at $\ox$ for  $w$, and $F$ is said to be semidifferentiable at $\ox$.
For convenience, we now recall in addition some notion from \cite[Theorem~2B.7]{DoR14}--an extended implicit mapping theorem that will be exploited in our next proof.
A mapping $\psi: \cP\times \X\to \X$ is Lipschitz continuous with respect to $p$ uniformly in $x$ around $(\op, \ox)$ with modulus $\ell>0$ if there exists $\epsilon>0$ such that
\begin{equation}\label{calp}
\big\|\psi(p, x) -\psi(p', x)\big\|\leq\ell\|p-p'\|\quad\mbox{ for all } \; p, p'\in \B_\epsilon(\op)\;\;
\mbox{ and all }\; x\in \B_\epsilon(\ox).
\end{equation}
A mapping $h: \X \to \X$ is a strict estimator of $\psi$ with respect to $x$ uniformly in $p$ at $(\op, \ox)$ with constant $\tau>0$ if $h(\ox) = \psi(\op, \ox)$ and there exists $\epsilon>$ such that
\begin{equation*}\label{sest}
\|r(p, x) - r(p, x')\|\leq \tau\,\|x-x'\|\;
\textrm{ for all }\; x, x'\in \B_\epsilon(\ox)\; \textrm{ and all }\; p\in \B_\epsilon(\op),
\end{equation*}
where $r(p, x)= \psi(p, x)-h(x)$.

\begin{Theorem}\label{thm:GE}
Let $(\op, \ox)\in \cP\times\X$ and $\psi: \cP \times \X \to \X$ satisfy the following conditions:
\begin{itemize}[noitemsep,topsep=2pt]
\item [{\rm (i)}] for any $p\in \cP$ near $\op$, the mapping $x\mapsto\psi(p, x)$ is $\C^1$-smooth around $\ox$ and $\nabla_x\psi(\cdot, \cdot)$ is jointly continuous at $(\op, \ox)$; and

\item [{\rm (ii)}] $\psi$ is Lipschitz continuous with respect to $p$ uniformly in $x$ around $(\op, \ox)$.
\end{itemize}
Assume that $f:\X\to\oR$ satisfies the hypotheses in Proposition~{\rm\ref{prop:GE}},   $-\psi(\op, \ox) \in \ri \sub f(\ox)$ and 
that \eqref{ge1} holds for $\opsi= \psi(\op, \cdot)$. 
Then, there exist a neighborhood $U$ of $\ox$ and a positive constant $\ve$  such that $U\cap S(p)$ is a singleton for all $p\in \B_{\ve}(\op)$ and the mapping $s:=U\cap S(\cdot)$ is $\M$-valued and Lipschitz continuous in $\B_\ve(\op)$,  where the solution mapping $S$ is taken from \eqref{sol}.
Moreover,  
\begin{itemize}[noitemsep,topsep=2pt]
\item [{\rm (a)}] if $\psi$ has a strong partial semiderivative with respect to $p$ at $(\op, \ox)$ in the sense that its partial graphical derivative $D_p\psi(\op, \ox):= D(\psi(\cdot, \ox))(\op): \P \tto \X$ is everywhere single-valued and  for every $\nu>0$ there exists $\epsilon>0$ such that
\begin{equation}\label{strongB}
\|\psi(\op + q, x) - \psi(\op, x) - D_p\psi(\op, \ox)(q)\| \leq \nu\|q\|
\end{equation}
for all $x\in \B_{\epsilon}(\ox)$ and all $q\in\epsilon\B\subset \P$, then $s$ is semidifferentiable at $\op$  with
\begin{equation}\label{sdS}
Ds(\op)(q) = -DG(\ox, 0)^{-1}(D_p\psi(\op, \ox)(q)), \quad q \in \P,
\end{equation}
where $G:= \psi(\op, \cdot)+\sub f$ and $DG(\ox, 0)^{-1}$ is calculated by \eqref{Jloc1};
\item [{\rm (b)}] if $\psi$ is differentiable with respect to $p$ at $(\op, \ox)$ and for every $\nu>0$ there exists $\epsilon>0$ such that \eqref{strongB} with $\nabla_p\psi(\op, \ox)(\cdot)$ in the place of $D_p\psi(\op, \ox)(\cdot)$ is satisfied for all $x\in \B_{\epsilon}(\ox)$ and all $q\in\epsilon\B\subset \P$, then $s$ is differentiable at $\op$  with
\begin{equation*}
\nabla s(\op) (q)= -DG(\ox, 0)^{-1}(\nabla_p\psi(\op, \ox)(q)), \quad q \in \P,
\end{equation*}
where $G$ and $DG(\ox, 0)$ are defined in {\rm(}a{\rm)}. 
\end{itemize}
\end{Theorem}

\begin{proof}
It is clear from $-\psi(\op, \ox) \in \ri \sub f(\ox)$ that $\ox \in S(\op)$.
We shall utilize \cite[Theorem~2B.7]{DoR14} to establish the existence of a Lipschitz continuous single-valued localization of $S$ around $\op$ for $\ox$.
To begin with, recall that $\opsi = \psi(\op, \cdot)$ is $\C^1$-smooth around $\ox$ and $-\opsi(\ox) = -\psi(\op, \ox) \in \ri \sub f(\ox)$ and that the condition \eqref{ge1} is fulfilled.
We then infer from Proposition~\ref{prop:GE} that the inverse mapping $G^{-1}$ of $G:= \opsi +\sub f$ has a Lipschitz continuous single-valued localization around $0$ for $\ox$ that is $\M$-valued and $\C^1$-smooth around $0$.  
Let $U$ be a neighborhood of $\ox$ and $\epsilon'>0$ such that 
\begin{equation}\label{map_gamma}
\gamma(v):= U\cap G^{-1}(v)     
\end{equation}
is a singleton for all $v\in \epsilon'\B$ and the mapping $\gamma:\epsilon'\B\to U\cap \M$ is $\C^1$-smooth. 
Let $\kappa>0$ be the Lipschitzian constant of $\gamma$ in $\epsilon'\B$.
We now show that $\psi(\op, \cdot)$ is a strict estimator of $\psi$ with respect to $x$ uniformly in $p$ at $(\op, \ox)$ with a constant $\tau\in (0, \kappa^{-1})$.
Let $r(p, x):= \psi(p, x) -\psi(\op, x)$ for $(p, x)\in\cP\times \X$ and $\epsilon''>0$ be such that $\B_{\epsilon''}(\op)\subset \cP$.
Pick $p\in \B_{\epsilon''}(\op)$ and $x, x'\in \B_{\epsilon''}(\ox)$ arbitrarily and observe that
\begin{align*}
\|r(p, x)-r(p, x')\| &\leq \sup_{y\in [x,\, x']}\|\nabla_xr(p, y)\|\cdot\|x-x'\|\nonumber\\
&=\sup_{y\in [x,\, x']}\|\nabla_x\psi(p, y) - \nabla_x\psi(\op, y)\|\cdot\|x-x'\|
\leq \tau\|x-x'\|,
\end{align*}
where $[x, x']:= \{(1-t)x+tx'\, |\, 0\leq t\leq 1\}$ and
$\tau:= \sup_{p\in \B_{\epsilon''}(\op),\, y\in \B_{\epsilon''}(\ox)}\|\nabla_x\psi(p, y) - \nabla_x\psi(\op, y)\|$.
Recalling that $\nabla_x\psi$ is jointly continuous at $(\op, \ox)$, we can assume without loss of generality that $\tau<\kappa^{-1}$, for otherwise we could select a smaller $\epsilon''$.
Thus, $\opsi = \psi(\op, \cdot)$ is a strict estimator of $\psi$ with respect to $x$ uniform in $p$ at $(\op, \ox)$ with constant $\tau$ satisfying $\kappa\tau<1$.
Appealing now to \cite[Theorem~2B.7]{DoR14}, we conclude that the solution mapping $S$ from \eqref{sol} has a Lipschitz continuous single-valued localization around $\op$ for $\ox$.
We then find a neighborhood of $\ox$, which we denote by the same $U$ as above for simplicity, and $\ve>0$ such that the mapping $s:\cP\supset\B_\ve(\op)\to U$, defined by $s(p)=U\cap S(p)$, is single-valued and Lipschitz continuous on $\B_\ve(\op)$.
In view of \eqref{sol}, we get for all $p\in \B_\ve(\op)$ that
\begin{equation*}
   -(\psi(p, s(p))- \psi(\op, s(p))) \in (\opsi+\partial f)(s(p)),
\end{equation*}
or, equivalently, $s(p) \in G^{-1}(r(p, s(p)))$. 
Shrinking neighborhoods of $\ox$ and $\op$ if necessary, we can infer from \eqref{calp} that $\|r(p, s(p))\|\leq \ell\|p-\op\|< \epsilon'$ for all $p\in \B_\ve(\op)$.
Using \eqref{map_gamma}, we  arrive at $s(p) \in U\cap G^{-1}(r(p, s(p))) = \{\gamma(r(p, s(p)))\}\subset \M$ for all $p\in \B_\ve(\op)$.
In summary, we have showed that $S$ has a single-valued localization $s: \B_\ve(\op)\to U$ that is $\M$-valued and Lipschitz continuous on $\B_\ve(\op)$.

We now turn to differential stability properties of the mapping $s$ above. 
Due to Lipschitz continuity of $s$ on $\B_\ve(\op)$, to prove that $s$ is semidifferentiable (respectively, differentiable) at $\op$ it suffices to show that $Ds(\op)$ is single-valued (respectively, $Ds(\op)$ is single-valued and linear); see  \cite[Exercise~9.25(a)--(b)]{rw}.
Pick an arbitrary $q\in \P$ and assume that $w\in Ds(\op)(q)$. 
By definition, there are sequences $t_k\searrow0$, $q^k \to q$ in $\P$, and $w^k\to w$ in $\X$ with $\ox +t_kw^k =s(\op+t_kq^k)$.
The latter yields $\ox+t_kw^k\in U$ and 
\begin{align*}
0&\in \psi(\op+t_kq^k, \ox+t_kw^k)+\sub f(\ox+t_kw^k)\\
&= G(\ox+t_kw^k)  +\psi(\op+t_kq^k, \ox+t_kw^k)-\psi(\op, \ox+t_kw^k)
\end{align*}
Let $\nu>0$ be arbitrarily small.
By {\rm (a)},   $\psi$ has a strong partial semiderivative with respect to $p$ at $(\op, \ox)$. So, we get from the above inclusion that
\begin{equation*}
-D_p\psi(\op, \ox)(q^k) \in G(\ox+t_kw^k)/t_k +\nu\|q^k\|\B
\end{equation*}
for all $k$ sufficiently large.
Passing to the limit as $k\to\infty$ and observing that $\nu$ was arbitrary   give us 
$-D_p\psi(\op, \ox)(q) \in DG(\ox, 0)(w)$,
or, equivalently, $w \in DG(\ox, 0)^{-1}(-D_p\psi(\op, \ox)(q))$.
It follows from  \eqref{ge1} and Proposition~\ref{prop:GE} that $DG(\ox, 0)^{-1}$ is calculated by \eqref{Jloc1}.
Thus, $DG(\ox, 0)^{-1}$ is single-valued and linear. We then have
\begin{equation*}
Ds(\op)(q)=\{w\} = -DG(\ox, 0 )^{-1}(D_p\psi(\op, \ox)(q)).
\end{equation*}
Since $q\in \P$ was chosen arbitrarily, the latter verifies semidifferentiability of $s$ at $\op$ and  \eqref{sdS}.
For {\rm (b)}, we can  argue similarly to drive  
\begin{equation*}
Ds(\op)(q)=\{w\} = -DG(\ox, 0 )^{-1}(\nabla_p\psi(\op, \ox)(q)) \quad\mbox{ for all }\; q\in \P.
\end{equation*}
Since $DG(\ox, 0)^{-1}$ is single-valued and linear,
the latter justifies differentiability of $s$ at $\op$ and the claimed formula for $\nabla s(\op)$ and hence completes the proof.
\end{proof}


\section{Asymptotics of Stochastic Programs with $\C^2$-Partly Smooth Regularizers}\label{sect:asym}
Consider the generic regularized stochastic problem
\begin{equation}\label{P}
\mini_{x\in \X}\; f(x):= \varphi(x) + \theta(x), \quad\mbox{ with }\; 
\varphi(x):= \E[\hph(x; \Z)],\tag{P}
\end{equation}
where $\Z$ is a random vector whose probability distribution $\Prob$ is supported on a measurable space $(\Xi, \A)$ and stands for a data sample and $\hph:\X\times \Xi \to \oR$ is a loss function.
The regularizer $\theta: \X\to \oR$ in \eqref{P} may either encode a geometric constraint or  promote desirable low-complexity structure, such as sparsity or low-rank, for a minimizer.
We assume that the function $\varphi(\cdot) = \E[\widehat \ph (\cdot; \Z)]$, is well-defined within some given domain, where the expectation $\E$ is taken with respect to $\Z\sim\Prob$ on $(\Xi, \A)$.
Let $\U$ be a bounded domain in $\X$, where we expect that a solution $\ox$ to the true problem \eqref{P} exists.
For each $x\in \U$,  the expected value $\varphi(x)$ can be approximated through a {\em sample average approximation} (SAA) estimator.
Namely, given $k$ data samples $\Z_1, \ldots, \Z_k$, independent and identically distributed (iid) copies of the random vector $\Z$, we   use the averaging function of $\hph(\cdot; \Z_i),\, i = 1, \ldots, k,$ in  place of $\varphi$ in \eqref{P} and then consider the so-called {\em SAA problem} of \eqref{P}, which is formulated as
\begin{equation}\label{aP}
\mini_{x\in \X}\; f_k(x):= \frac{1}{k}\sum_{i=1}^k\hph(x; \Z_i) + \theta(x).\tag{P$_k$}
\end{equation}
The emphasis of this section is on asymptotics of sequences of solutions to the SAA problem \eqref{aP}.
Our analysis closely follows developments of sensitivity analysis and asymptotic theory made in \cite{KiR92, KiR93} by King and Rockafellar  for (stochastic) generalized equations under {proto-differentiability}, {\em subinvertibility}, and single-valuedness of the inverse graphical derivatives; see \cite[Assumptions~M.2--M.4]{KiR93} and discussions therein.
According to \cite[Section~3]{KiR92}, a set-valued mapping $F:\X\tto\Y$ is subinvertible at $(\ox, 0)$ if $0\in F(\ox)$ and there exist a compact convex neighborhood $U\subset \X$ of $\ox$, a positive constant $\ve$, and a nonempty convex-valued mapping $G:\ve\B\tto U$ such that $\ox\in G(0)$ and $\gph G$ is a closed subset of $(\ve\B\times U)\cap\gph F^{-1}$.
If $F^{-1}$ admits a continuous selection $x(y)$ around $0$ with $x(0) = \ox$, then $F$ is clearly subinvertible at $(\ox, 0)$.
Assume now that $F:\X\tto\X$ is a maximal monotone mapping; see  \cite[Definitions~12.1 and 12.5]{rw} for the definition of this concept.
It was shown in \cite[Theorem~5.1]{KiR92} that the condition $DF(\ox, 0)^{-1}(0)=\{0\}$ is sufficient for subinvertibility of $F$ at $(\ox, 0)$.
Recalling also from \cite[Theorem~12.65]{rw} that the single-valued property of $DF(\ox, 0)^{-1} = DF^{-1}(0, \ox)$ is equivalent to semidifferentiability of $F^{-1}$ at $0$ for $\ox$. Recall that 
Theorem~\ref{thm:GE} provides sufficient conditions for semidifferetiability of solution mappings to generalized equations  associated with $\C^2$-partly smooth functions. 
In order to exploit our analysis in Section~\ref{sect:sTED} in analyzing asymptotics of the stochastic generalized equation associated with \eqref{P}, we focus on objective functions that consist of a $\C^2$-smooth loss function $\varphi$ in conjunction with a $\C^2$-partial smoothness regularizer $\theta$.
Fixing these settings, we are going to study the asymptotic distribution of solutions to
\begin{equation}\label{sge}
0 \in \frac{1}{k}\sum_{i=1}^k \nabla_x\hph(x; \Z_i) +\sub \theta(x),
\end{equation}
which is referred to as the SAA generalized equation of the true variational system corresponding to problem \eqref{P} given by
\begin{equation}\label{ge}
0 \in \E[\nabla_x\hph(x; \Z)] + \sub\theta(x).
\end{equation}
To begin with, we formally state standing assumptions in our analysis.

\begin{Assumption}[integrability and smoothness]\label{ass:S}
\begin{itemize}[noitemsep,topsep=2pt]
\item [(S.1)] The function $\hph:\U \times \Xi\to\oR$ is measurable in the second variable $\xi\in \Xi$ for each $x\in \U$, a bounded domain in $\X$. 

\item [(S.2)]  For almost every $\xi\in\Xi$, the function $x\mapsto\hph(x; \xi)$ is $\C^2$-smooth on $\U$.

\item [(S.3)] There exists $g: \Xi\to \R_+$ with $\E[g(\Z)]<\infty$ such that
\begin{equation*}
\max\big\{|\hph(x; \xi)|,\, \|\nabla_x\hph(x; \xi)\|, \|\nabla^2_{xx}\hph(x; \xi)\|\big\}\leq g(\xi)
\end{equation*}
 for almost every $\xi\in \Xi$ and for all $x\in \U$.

\item [(S.4)] There exists $\ell: \Xi \to \R_{++}$ with $\E[\ell(\Z)]<\infty$ satisfying
\begin{equation*}
\max\big\{|\hph(x; \xi) - \hph(x'; \xi)|,\, \|\nabla_x\hph(x; \xi) - \nabla_x\hph(x'; \xi)\|\big\} \leq \ell(\xi)\|x-x'\|
\end{equation*}
for almost every $\xi\in \Xi$ and for all $x, x'\in \U$.
\end{itemize}
Here, in {\rm (S.3)--(S.4)}, the gradients $\nabla_x\hph(\cdot; \xi)$ and Hessians $\nabla^2_{xx}\hph(\cdot; \xi)$  of $\hph(\cdot; \xi)$ exist for almost every $\xi\in \Xi$, by {\rm (S.2)}.
\end{Assumption}

Note from assumptions {\rm (S.1)--(S.4)} and \cite[Theorem~9.64]{SDR09} that the expectation function $\varphi = \E[\hph(\cdot; \Z)]$ is well-defined and, moreover, twice continuously differentiable with
\begin{equation*}
\nabla\varphi(x) = \E[\nabla_x\hph(x; \Z)]\quad\mbox{ and }\quad
\nabla^2\varphi(x) = \E[\nabla^2_{xx} \hph(x; \Z)], \quad x\in \U.
\end{equation*}

\begin{Assumption}[analytical assumptions]\label{ass:A}
Assume that $\ox\in \dom \th$ for which the following conditions are satisifed:
\begin{itemize}[noitemsep,topsep=2pt]
\item [(A.1)] $-\nabla\varphi(\ox) \in \ri\sub \theta (\ox)$.

\item [(A.2)] $\theta$ is $\C^2$-partly smooth at $\ox$ relative to a $\C^2$-smooth manifold $\M$ with the local representation \eqref{mfold} and $\widehat\theta: \X\to\R$ is a $\C^2$-smooth representation of $\theta$ around $\ox$ relative to $\M$; $\theta$ is  prox-regular and subdifferentially continuous at $\ox$ for $-\nabla\varphi(\ox)$.

\item [(A.3)] The second-order condition  
\begin{equation}\label{A.3}
\left\{w \in T_\M(\ox)\, \big|\, \nabla^2_{xx} L(\ox, \omu)(w)\in N_\M(\ox)\right\} = \{0\}
\end{equation}
holds, where $L(x, \mu):= \varphi(x) + \widehat\theta(x) + \la \mu, \Phi(x)\ra$ for any $(x, \mu)\in \X\times \R^m$ and where $\omu\in \R^m$ is the unique vector, satisfying $\nabla\Phi(\ox)^*\omu = - (\nabla\varphi(\ox) + \nabla \widehat \theta (\ox))$.
\end{itemize}
\end{Assumption}

Let $\P$ be a separable Banach space equipped with its Borel algebra $\mathscr{B}(\P)$.
Recall from  \cite[Section~3]{Kin89} that a sequence $\{{\bf p}_k\}_{k\in\N}$ of random vectors ${\bf p}_k: \Xi \to \P$ is said to satisfy a {\em generalized central limit} formula if there exist a limit $\bar p\in \P$, a sequence $\{t_k\}_{k\in\N}$ of positive scalars decreasing to $0$, and a limit distribution $q$ such that
$({\bf p}_k - \bar p)/t_k \xrightarrow{\D} q$, where the notation $\xrightarrow{\D}$ stands for convergence in distribution.
The following lemma establishes asymptotic behavior of the sequence of stochastic gradient mappings
\begin{equation}\label{sg}
\nabla\hph_k:= \tfrac{1}{k}\sum_{i=1}^k\nabla_x\hph(\cdot; \Z_i), \quad k\in \N,
\end{equation}
taken from \eqref{sge}.
In what follows, we use $\C^1(\U, \X)$ to denote the space of $\C^1$-smooth mappings $\phi:\U\to\X$ endowed with the norm 
\begin{equation*}
\|\phi\|_{1,\, \U}: = \sup_{x\in \U}\|\phi(x)\|+\sup_{x\in \U}\|\nabla\phi(x)\|,
\end{equation*}
where $\phi(x)\in \X$ and $\nabla\phi(x) \in \L(\X, \X)$, the space of linear operators from $\X$ into $\X$ itself equipped with the operator norm.
And, as usual, $\mathscr{B}(\X)$ and $\mathscr{B}(\C^1(\U, \X))$ stand for the Borel algebras on $\X$ and $\C^1(\U, \X)$, respectively.

\begin{Lemma}\label{lem:CL}
Suppose that $\hph:\X \times \Xi:\to \oR$ satisfies Assumption~{\rm\ref{ass:S}}.
Then, for each $k$, the stochastic gradient mapping $\nabla\hph_k$ in \eqref{sg} is a random vector in $\C^1(\U, \X)$, and $\nabla\hph_k$ converges to $\nabla\varphi$ with probability one   in $\C^1(\U,\X)$.
Moreover, one has
\begin{equation}\label{CL}
\sqrt{k}\big(\nabla\hph_k - \nabla\varphi)\xrightarrow{\D}\phi 
\end{equation}
as random vectors in $\C^1(\U, \X)$, where $\phi$ is normally distributed with mean $0$ and variance 
\begin{equation*}
\Sigma(x) :=\Var \nabla_x\hph(x; \Z)= \E[(\nabla_x\hph(x; \Z) - \nabla\varphi(x))(\nabla_x\hph(x; \Z) - \nabla\varphi(x))^\top], \quad  x\in \U.
\end{equation*}
\end{Lemma}

\begin{proof}
Note that $\nabla\hph_k$ converging  with probability one to $\nabla\varphi$ in $\C^1(\U, \X)$ means that $\nabla\hph_k(x)$ and $\nabla^2\hph_k(x):= \tfrac{1}{k}\sum_{i=1}^k\nabla^2_{xx}\widehat\varphi(x; \Z_i)$ converge with probability one to $\nabla\varphi(x)$ in $\X$ and to $\nabla^2\varphi(x)$ in $\L(\X, \X)$, respectively, uniformly in $\U$.
Both assertions are an adaptation of \cite[Theorem~9.60]{SDR09} for $\nabla_x\hph(\cdot; \xi)$ and $\nabla^2_{xx}\hph(\cdot; \xi)$ that are finite-dimensional vector-valued mappings.
The claimed generalized central limit formula in \eqref{CL} follows from  \cite[Theorem~A.3]{KiR93}.
\end{proof}

We conclude this section with asymptotic distribution of the solution to the SAA generalized equation in \eqref{sge}.

\begin{Theorem}[asymptotic distribution]\label{thm:ad}
Suppose that Assumptions~{\rm\ref{ass:S}} and {\em\ref{ass:A}} hold for $\hph$ and $\varphi = \E[\hph(\cdot; \Z)]$ with respect to a bounded domain $\U$ in $\X$ containing $\ox$. 
Then, there exists a neighborhood $U\subset\U$ of $\ox$ such that, with probability one, for $k$ sufficiently large, the SAA generalized equation in \eqref{sge} admits a unique local solution ${\bf x}^k \in U$. 
Moreover, the sequence $\{{\bf x}^k\}_{k\in\N}$ converges to $\ox$ with probability one and satisfies the generalized central limit formula
\begin{equation}\label{CL1}
\sqrt{k}({\bf x}^k -\ox) \xrightarrow{\D} -\big(P_{T_\M(\ox)}\circ\nabla^2_{xx}L(\ox, \omu)\vert_{T_\M(\ox)}\big)^{-1}(P_{T_\M(\ox)}(\phi(\ox))),
\end{equation}
where $\phi\in \C^1(\U, \X)$ is taken from Lemma~{\rm\ref{lem:CL}} and $L$ and $\omu$ are from Assumption~{\rm\ref{ass:A}(A.3)}.
\end{Theorem}

\begin{proof}
Let $\P=\C^1(\U, \X)$ and $\cP\subset \P$ an open subset containing $\op:=\nabla\varphi$.
Consider the evaluation mapping $\psi: \cP\times \U \to \X$, defined by $\psi(p, x) := p(x)$ for any $(p,x)\in \cP\times \U$.
This instance of $\psi$ clearly satisfies assumptions {\rm (i)} and {\rm (ii)} in Theorem~\ref{thm:GE} at $(\op, \ox)$.
It was also pointed out in \cite[Remark~4.2]{KiR92} that $\psi$ satisfies the assumption {\rm (a)} in Theorem~\ref{thm:GE} at $(\op, \ox)$  with
\begin{equation}\label{ad}
D_p\psi(\op, \ox)(q) = q(\ox)\quad \mbox{ for all }\; q\in \C^1(\U, \X).
\end{equation}
Let $U\subset \U$ and $\delta>0$ be the neighborhood of $\op$ and the constant taken from Theorem~\ref{thm:GE}.
In the event of the convergence $\nabla\hph_k \to \nabla \varphi$ in $\C^1(\U, \X)$, which happens with probability one as shown in Lemma~\ref{lem:CL}, Theorem~\ref{thm:GE} tells us that \eqref{sge} eventually admits a unique local solution ${\bf x}^k \in U$.
Compactness of $U$ implies that the sequence $\{{\bf x}^k\b$ has at least one cluster point, says $\widehat {\bf x}$. Thus, there is a subsequence 
$\{{\bf x}^{k_i}\}_{i\in \N}$ of $\{{\bf x}^k\b$ that ${\bf x}^{k_i}\to \widehat {\bf x}$. 
It follows from the local Lipschitz continuity  of the localization $s(\cdot) = U\cap S(\cdot)$ of the solution mapping
\begin{equation*}
S(p):= \big\{x\in \X\, \big|\, 0\in \psi(p, x)+\sub\theta(x)\big\}, \quad p \in \cP,
\end{equation*}
on $\B_\delta(\nabla\varphi)$ that 
\begin{equation*}
\widehat {\bf x} = \lim_{i\to \infty} {\bf x}^{k_i} = \lim_{i\to \infty}s(\nabla\hph_{k_i})  {=} s(\nabla\varphi) = \ox,
\end{equation*}
where the third equality holds with probability one. 
This tells us that all the convergent subsequences of $\{{\bf x}^k\}_{k\in \N}$ tend to $\ox$ with probability one and therefore the  sequence $\{{\bf x}^k\}_{k\in \N}$ converges to $\ox$ with probability one.
We can deduce from Theorem~\ref{thm:GE} together with \eqref{ad}  that the graphical localization $s$ of $S$ is almost surely semidifferentiable at $\op = \nabla \varphi$ for $\ox$ with
\begin{equation*}
Ds(\op)(q) = -\big(P_{T_\M(\ox)}\circ\nabla^2_{xx}L(\ox, \omu)\vert_{T_\M(\ox)}\big)^{-1}\big(P_{T_\M(\ox)}(q(\ox))\big), \quad q\in \C^1(\U, \X);
\end{equation*}
see \cite[Definition~3.1]{Kin89} for the definition of the  latter differentiability notion.
Applying \cite[Theorem~3.2]{Kin89} to the continuous mapping $s: \B_\delta(\nabla\varphi)\to U$ and the sequence $\{\nabla\hph_k\}_{k\in \N}$ that satisfies \eqref{CL} by Lemma~\ref{lem:CL}, we  obtain the generalized central limit formula \eqref{CL1}, which completes the proof.
\end{proof}

Note that  asymptotic distribution of SAA of a generalized equation was first established in \cite[Theorem~2.7]{KiR93} under subinvertibility and the single-valuedness of the inverse of the graphical derivative of the solution mapping; see also \cite[section~5.2.2]{SDR09} for a similar result for the KKT system of nonlinear programming problems. These assumptions are automatically satisfied under Assumption~\ref{ass:A}. Furthermore, the latter result assumes the almost sure convergence of the sequence $\{{\bf x}^k\}_{k\in\N}$ in Theorem~\ref{thm:ad}, a fact that was not assumed in our result and indeed was proven as a consequence of Assumption~\ref{ass:A}. The asymptotic distribution of SAA of a generalized equation was recently achieved in \cite[Theorem~3.1]{DDJ23} when the solution mapping to the generalized equation has a ${\cal C}^1$ single-valued graphical localization. According to Theorem~\ref{thm:GE}, our Assumption~\ref{ass:A} ensures such a property for the solution mapping to \eqref{ge}. However, similar to \cite[Theorem~2.7]{KiR93}, \cite[Theorem~3.1]{DDJ23} assumes the almost sure convergence of the sequence $\{{\bf x}^k\}_{k\in\N}$.
Note that while our framework is not as general as the one considered in \cite{DDJ23}, our approach clearly can be used for a generalized equation. 
Moreover, we can replace Assumption~\ref{ass:A} with the mere assumption that 
$\sub \th$ is strictly proto-differentiable and study the  asymptotic distribution of solutions to the SAA approximation of the generalized equation 
$$
0\in f(x) +\sub \th(x),
$$
where $f:\X\to \R$ is ${\cal C}^1$-smooth. Such a condition is weaker than assuming the solution mapping to the generalized equation has a ${\cal C}^1$ single-valued graphical localization, which was used in \cite{DDJ23}. We will pursue this issue in our future works.


\end{document}